\documentclass{amsart}
\usepackage{amssymb}

\usepackage{graphicx}
\usepackage{subcaption}

\newtheorem{theorem}{Theorem}[section]
\newtheorem{lemma}[theorem]{Lemma}

\theoremstyle{definition}

\theoremstyle{remark}
\newtheorem{remark}[theorem]{Remark}

\numberwithin{equation}{section}

\newtheorem{proposition}[theorem]{Proposition}
\newtheorem{corollary}[theorem]{Corollary}
\newcommand\ch{{\mathcal{C}_h}}
\newcommand\ehx{{\mathcal{E}^x_h}}
\newcommand\ehy{{\mathcal{E}^y_h}}
\newcommand\eh{{\vec{\mathcal{E}_h}}}

\usepackage{amsmath}
\usepackage{multirow}
\usepackage{algorithm}
\usepackage{amsrefs}
\usepackage{color}
\usepackage{optidef}

\usepackage[colorlinks=true,linkcolor=red,citecolor=green,filecolor=magenta,urlcolor=cyan]{hyperref}
\usepackage[noabbrev]{cleveref}
\crefname{proposition}{Proposition}{Propositions}

\begin{document}

\title[A Linearly Convergent Solver for the AC Equation with FH Potential]
{An Unconditionally Linearly Convergent ADMM Approach for the Allen-Cahn Equation with Flory-Huggins Potential}


\author{Peng Jiang}
\address{Corresponding author. School of Mathematical Sciences, Peking University, Beijing, China}
\curraddr{}
\email{2100012131@stu.pku.edu.cn}

\author{Shengtong Liang}
\address{School of Mathematical Sciences, Peking University, Beijing, China}
\curraddr{}
\email{liangshengtong@pku.edu.cn}

\author{Tiao Lu}
\address{HEDPS, CAPT, LMAM, School of Mathematical Sciences, Peking University, Beijing, China}
\curraddr{}
\email{tlu@pku.edu.cn}

\thanks{This work was partially supported by the National
Natural Science Foundation of China (Grant No. 12288101)}

\thanks{The authors thank Prof. Ruo Li of the School of Mathematical Sciences at Peking University, and Prof. Zhonghua Qiao of the Department of Applied Mathematics at Hong Kong Polytechnic University, for their valuable discussions.}

\dedicatory{}

\date{}



\begin{abstract}
The Allen-Cahn equation with Flory-Huggins potential is a fundamental and crucial model in phase field simulation for describing phase separation phenomena,
which serves as a core tool in diverse branches of natural sciences.
The numerical simulation of the Allen-Cahn equation is of great importance but poses significant challenges due to the strong nonlinearity and the presence of logarithmic singularities at $u=0,1$ in the Flory-Huggins potential.
In this paper, we consider convex splitting schemes to 
guarantee unconditional unique solvability,
which reduces the numerical simulation to solving a singular nonlinear system arising from spatial discretization at each time step.
We propose an iterative solver that is specifically designed for such systems based on the alternating direction method of multipliers (ADMM) approach. 
The scheme possesses properties such as bound preserving and discrete energy stability.
Building upon the recent unconditionally convergent ADMM framework for the Cahn-Hilliard equation (Li et al., 2026), our key theoretical contributions are twofold:
(a) a proof of unconditional convergence when the multiplier update step size $\alpha \in (0,\frac{\sqrt{5}+1}{2})$;
(b) a rigorous establishment of the linear convergence for the embedded ADMM solver.
This effectively liberates the solver from time-step constraints or strict separation conditions.
Comprehensive numerical experiments validate our proposed ADMM framework, where its theoretical predictions are fully substantiated in practice, showcasing efficiency and robustness.
\end{abstract}

\maketitle
\tableofcontents
\section{Introduction}
\label{section:introduction}
Phase field models \cites{chen2002phase,steinbach2013phase} were introduced to describe the diffusive nature of interfaces between two phases. 
It is known that the Allen-Cahn equation \cite{allen1979microscopic} serves as a fundamental phase field model. 
In this paper, we consider the Allen-Cahn equation of the form
\begin{equation}\label{eq:ac}
      \begin{cases}
    \frac{\partial u}{\partial t}-\epsilon^2\Delta u+f(u)=0, &\qquad (x,t)\in \Omega \times (0,+\infty),\\
    u|_{t=0}=u_0,&\qquad x\in \Omega,
\end{cases}
\end{equation}
where the domain $\Omega \subset \mathbb{R}^d$ with $d=2,3$. $u=u(x,t)$ is the concentration of a phase in a two-phase system and $u_0=u_0(x)$ is the initial data. 
The parameter $\epsilon>0$ is a constant which governs the typical length scale of an interface in the dynamical evolution \cites{wang1993thermodynamically,wheeler1992phase}. 
The nonlinear term $f(u)$ is the derivative of a potential function $F(u)$. 

Though the double-well potential is simpler because of its polynomial form, the Flory-Huggins potential is more realistic from the physical viewpoint \cite{flory1942thermodynamics}. 
The logarithmic Flory-Huggins free energy potential is of the form
\begin{equation}
    F(u)=u\log(u)+(1-u)\log(1-u)+\theta(u-u^2),
\end{equation}
where $\theta>2$ is the energy parameter, and respectively,
\begin{equation}
    f(u)=\log(u)-\log(1-u)+\theta(1-2u).
\end{equation}
Define the continuous energy functional,
\begin{equation}
    E(u)=\int_{\Omega} \left(F(u)+\frac{\epsilon^2}{2}|\nabla u|^2\right)\text{d}x.
\end{equation}
We then obtain the following natural energy dissipation property \cite{allen1979microscopic}.
\begin{proposition}[Continuous energy stability]\label{prop:sta}
    The energy of a smooth solution $u$ to \eqref{eq:ac} is dissipated with time.
\end{proposition}
\begin{proof}
    By definition, 
    \begin{equation}
        \begin{aligned}
        \frac{\partial E(u)}{\partial t}
        &=\int_{\Omega}(f(u) - \epsilon^2\Delta u)\frac{\partial u}{\partial t}\text{d}x
        &=-\int_{\Omega}\left|\frac{\partial u}{\partial t}\right|^2\text{d}x
        &\leq 0.
    \end{aligned}
    \end{equation}
\end{proof}
Since this equation plays a key role in phase field models and phase transition simulations, 
theoretical analysis and numerical studies concerning it are particularly important. 
An ideal numerical algorithm for the simulation of Allen-Cahn equation with Flory-Huggins potential should possess the following desirable properties: 
(a) bound preserving property: the logarithmic term confines the solution to $0<u<1$,  so the numerical scheme should preserve this bound; 
(b) discrete energy stability: corresponding to the continuous energy stability \cite{provatas2011phase}. 
In what follows, we review some recent works on numerical simulations of Allen-Cahn equation.

In \cites{guillen2013linear,yang2017numerical},
the so-called invariant energy quadratization (IEQ) method was proposed,
and subsequently, the scalar auxiliary variable (SAV) approach \cite{shen2018scalar} was developed as a significant variant.
These methods have been widely applied in numerous numerical simulation scenarios, such as molecular beam epitaxy (MBE) \cite{yang2017numerical} and crystal growth \cites{zhao2017numerical,zhao2018general}.
The core idea behind this class of methods is to introduce auxiliary variables (spatially dependent in IEQ, or purely scalar in SAV) to transform the original free energy functional into a quadratic form.
Despite their significant success in numerical simulations, these methods share two limitations:
(a) the scheme ensures stability only for the transformed discrete energy, but may not for the discrete energy of the original formulation
\cites{yang2017numerical,wang2020stabilized}, which may lead to non-physical numerical oscillations;
(b) the time step size required to satisfy the discrete maximum bound principle (MBP) \cites{chen1992generation,nishiura1987stability} is restricted by a uniform constant,
which introduces a trade-off between stability and computational efficiency.

Many numerical schemes achieve bound preserving properties for the numerical solution by satisfying 
the MBP \cites{tang2016implicit,du2021maximum,stehlik2015maximum,shen2016maximum}.
Reference \cite{du2021maximum} introduced an exponential time differencing (ETD) method for a class of semilinear parabolic equations 
including the nonlocal Allen-Cahn equation \cites{du2019maximum,kroemer2025quantitative}. 
The scheme is shown to preserve the MBP unconditionally and exhibit asymptotic compatibility. 
However, this analysis is restricted to the Allen-Cahn equation with the double-well potential, which is well-defined and continuous for all $u\in \mathbb{R}$.
Notably, for the Flory-Huggins potential, the algorithm faces limitations, 
because the validity of the discrete MBP requires the potential term to be bounded, 
a property that the Flory-Huggins potential does not satisfy. 
Furthermore, the MBP is only theoretically guaranteed, which means, 
in numerical practice, we need the solution to strictly satisfy the separation property, i.e., $\|u\|_{\infty}\leq 1-\delta$ for sufficiently small $\delta$.

Reference \cite{li2022stability} proposed an algorithm based on operator splitting methods \cite{macnamara2017operator}. 
It applies the Strang splitting method \cite{strang1968construction} for second-order temporal discretizations. 
The merits of the proposed algorithm include its simplicity, its preservation of the MBP and discrete energy dissipation, as well as its proven convergence. 
Nonetheless, the method exhibits a shortcoming: the MBP is contingent upon a time step restriction. 
Consequently, this imposes challenges for long-time simulations and for computations employing large time steps.

A recent work \cite{wang2020stabilized} introduced the stabilized energy factorization (SEF) approach. 
Its core idea is to factorize the energy functional into a product of several factors  \cite{kou2020novel}, 
with each factor handled individually based on its unique property. 
This approach offers multiple benefits: linear semi-implicit scheme, preservation of the discrete MBP and discrete energy stability, as well as convergence property. 
However, in the paper, the stability parameter condition required for the MBP is stringent and difficult to choose in practical computations \cite{wang2020stabilized}. 
Moreover, the algorithm struggles to fully overcome the numerical instability caused by the singularity at $u=0,1$.


The convex splitting scheme was initially proposed in the study of unconditionally energy stable schemes for the Cahn-Hilliard equation by Eyre \cite{eyre1998unconditionally}, 
and was subsequently successfully extended and applied to various phase field models \cites{chen2012linear,shen2012second,baskaran2013convergence}. 
The fundamental principle of the convex splitting approach is to express the energy functional as the sum of a convex and a concave contribution. 
Correspondingly, the convex part is handled implicitly and the concave part explicitly, 
leading to a semi-implicit scheme known as a convex splitting scheme, 
which possesses several desirable properties, such as unique solvability and discrete energy stability \cites{eyre1998unconditionally,du2020phase}. 

However, while extensive efforts have been devoted to designing such unconditionally energy-stable temporally discrete schemes, a crucial computational obstacle remains largely unresolved: efficiently solving the resulting highly nonlinear algebraic systems at each time step.
The lack of a robust, theoretically guaranteed iterative solver often compromises the unconditional stability claimed by these schemes in practical simulations.
Specifically, applying the convex splitting method to the Allen-Cahn equation with Flory-Huggins potential results in a deeply complex nonlinear discrete system.
The primary mathematical difficulty arises from the logarithmic singularities of the potential. As the phase variable $u$ approaches the physical bounds $0$ or $1$, the derivative of the potential becomes unbounded.
Consequently, the gradient of the corresponding objective functional lacks global Lipschitz continuity, making standard convergence rate analyses for iterative optimization algorithms inapplicable.

Li et al. \cite{li2025overcoming} developed an iterative method with the use of a variant of ADMM to solve the singular nonlinear systems 
arising from the application of the convex splitting scheme to the Cahn-Hilliard equation with Flory-Huggins potential.
The method possesses properties including unconditional convergence, bound preservation, and mass conservation. 
If the strict separation property is satisfied, i.e., there exists a uniform positive lower bound on the distance 
from the iterative sequence $\{u_2^{(k)}\}$ generated by the iterative solver to the two singular points at a point-wise level,
then the iterative algorithm exhibits a linear convergence rate. 
Nevertheless, the assumption is strong, artificially imposed, and not natural.

In this paper, adapting the strategy developed in \cite{li2025overcoming}, we present a novel iterative solver for the nonlinear discrete system at each time step.
This system naturally arises from the application of the convex splitting scheme to the Allen-Cahn equation with Flory-Huggins potential.
Compared with existing methods such as IEQ, ETD, operator splitting and SEF, 
our proposed new algorithm naturally satisfies the core requirements of bound preservation and discrete energy stability.
Moreover, its linear convergence requires no restrictions on the time step or the separation condition.
It effectively circumvents the numerical instability caused by the logarithmic singularity of the Flory-Huggins potential, 
thus overcoming the limitations of current methods, which often struggle to balance stability and computational efficiency and lack guaranteed convergence.

We reformulate the nonlinear system equivalently as a convex optimization problem for finding a minimizer.
To solve this optimization problem, we employ the alternating direction method of multipliers (ADMM) framework \cite{neal2011distributed}.
More specifically, the objective functional is decomposed into a linear part and a logarithmic part.
The subproblem corresponding to the linear part is solved efficiently with the fast Fourier transform (FFT) techniques \cite{cooley1965algorithm}, due to its circulant coefficient matrix.
For the logarithmic part, the corresponding subproblem is reduced to solving multiple one-dimensional equations with a monotonic derivative, which enables us to apply a safeguarded Newton's method to each equation \cites{dennis1996numerical,ortega2000iterative} in parallel.

The technical difference between the two works is as follows. 
The Cahn-Hilliard equation is a fourth-order equation, which makes the nonlinear system from the convex splitting scheme ill-conditioned.
In order to treat the Cahn-Hilliard equation as an $H^{-1}$ gradient flow, 
Li et al. \cite{li2025overcoming} introduced a variant of ADMM for solving a minimax problem. 
This variant is a new framework whose linear convergence result has not been extensively studied. 
In this paper, the Allen-Cahn equation is a second-order equation and is treated as an $L^{2}$ gradient flow. 
Therefore, the classical ADMM scheme for convex optimization problems can be directly adopted, 
and its various properties, including the linear convergence rate, have been well analyzed.

Our key theoretical contributions are twofold:
(a) We provide a rigorous proof of unconditional convergence for the ADMM solver for any multiplier update step size $\alpha \in (0,\frac{\sqrt{5}+1}{2})$.
(b) We establish the unconditional linear convergence of the solver. 
Unlike previous work on the Cahn-Hilliard equation that relied on a strict separation assumption to guarantee linear convergence \cite{li2025overcoming}, 
we show that the Allen-Cahn equation is inherently well-conditioned to achieve unconditional linear convergence under the proposed ADMM approach.

The remainder of this paper is organized as follows.
\Cref{section:discrete} introduces the discrete function spaces and operators within the cell-centered finite difference (CCFD) framework.
In \Cref{section:admm}, we detail the convex splitting scheme and formulate our novel ADMM-based iterative solver.
The core of our theoretical contribution is presented in \Cref{section:analysis}, where we rigorously establish the well-posedness, discrete energy stability, and the unconditional linear convergence of the proposed solver.
Comprehensive numerical experiments that validate our theoretical findings and illustrate the robustness of the solver are reported in \Cref{section:numerical}.
Finally, concluding remarks are given in \Cref{section:conclusion}.

\section{Discrete Function Spaces and Discrete Operators}
\label{section:discrete}
In this section, we introduce the discrete function spaces and discrete operators within the framework of CCFD method \cite{tryggvason2011direct}. 
The following notations are partially taken from \cites{chen2019positivity,li2025overcoming}.

We only present the two-dimensional case here, and the forms of the three-dimensional case are similar. 
The problem we consider in this paper is defined on a square domain $\Omega=(0,L)^2$ where $L>0$ is the width of domain, 
with periodic boundary conditions imposed. Under this configuration, the equations are cast into the following system form.
\begin{equation}
\begin{cases}
    \frac{\partial u}{\partial t}-\epsilon^2\Delta u+\log(u)-\log(1-u)+\theta(1-2u)=0,\  &(x,y)\in \Omega,\  t>0,\\
    u(x,y,0)=u_0(x,y),\ &(x,y)\in\Omega,\\
    u(x,y,t)=u(x+L,y,t)=u(x,y+L,t),\  &(x,y)\in \mathbb{R}^2,\  t>0.
\end{cases}
\end{equation}
For simplicity, we use a uniform division with the mesh size $h=\frac{L}{N}$, where $N\in \mathbb{N}_{+}$. 
In the CCFD method, we naturally define the uniform 1D grids for cell centers $C$ and cell edges $E$ as follows:
\begin{equation}
    C=\{p_i \mid i\in \mathbb{Z}\},
    \quad E=\{p_{i+\frac{1}{2}} \mid i\in \mathbb{Z}\},
\end{equation}
where $p_i=(i-\frac{1}{2})h$ and $p_{i+\frac{1}{2}}=ih$.
Then, we define four $2$D discrete $N^2$-periodic function spaces,
\begin{subequations}
    \begin{align}
    &\ch=\{\phi:C\times C\rightarrow \mathbb{R}\mid\phi_{i,j}=\phi_{i+\lambda N,j+\mu N},\forall i,j,\lambda, \mu\in\mathbb{Z}\},\\
    &\ehx=\{\phi:E\times C\rightarrow \mathbb{R}\mid\phi_{i+\frac{1}{2},j}=\phi_{i+\frac{1}{2}+\lambda N,j+\mu N},\forall i,j,\lambda, \mu\in\mathbb{Z}\},\\
    &\ehy=\{\phi:C\times E\rightarrow \mathbb{R}\mid\phi_{i,j+\frac{1}{2}}=\phi_{i+\lambda N,j+\frac{1}{2}+\mu N},\forall i,j,\lambda, \mu\in\mathbb{Z}\},\\
    &\eh=\ehx\times \ehy,
\end{align} 
\end{subequations}
where $\phi_{i,j}=\phi(p_i,p_j)$ is a grid function.

For $\phi \in \ch$, the discrete spatial difference operator in the $x$-direction is defined as
\begin{equation}
    \begin{aligned}
        D_x:\ch&\rightarrow\ehx,\\
        \phi&\mapsto(D_x \phi)_{i+\frac{1}{2},j}=\frac{1}{h}(\phi_{i+1,j}-\phi_{i,j}).
    \end{aligned}
\end{equation}
Similarly, for the $y$-direction
\begin{equation}
    \begin{aligned}
        D_y:\ch&\rightarrow\ehy,\\
        \phi&\mapsto(D_y \phi)_{i,j+\frac{1}{2}}=\frac{1}{h}(\phi_{i,j+1}-\phi_{i,j}).
    \end{aligned}
\end{equation}
The discrete gradient operator is defined as
\begin{equation}
    \begin{aligned}\nabla_h:\ch&\rightarrow\eh\\
    \phi&\mapsto(D_x\phi,D_y\phi).
    \end{aligned}
\end{equation}
For $\phi \in \ehx$, the discrete spatial difference operator in the $x$-direction is defined as
\begin{equation}
    \begin{aligned}
        d_x:\ehx&\rightarrow\ch,\\
        \phi&\mapsto(d_x\phi)_{i,j}=\frac{1}{h}(\phi_{i+\frac{1}{2},j}-\phi_{i-\frac{1}{2},j}).
    \end{aligned}
\end{equation}
Similarly, for $\phi \in \ehy$, the discrete spatial difference operator in the $y$-direction is defined as
\begin{equation}
    \begin{aligned}
        d_y:\ehy&\rightarrow\ch,\\
        \phi&\mapsto(d_y\phi)_{i,j}=\frac{1}{h}(\phi_{i,j+\frac{1}{2}}-\phi_{i,j-\frac{1}{2}}).
    \end{aligned}
\end{equation}
Thus for $\vec{f}=(f^x,f^y)\in\eh$, the discrete divergence operator can be written as\begin{equation}
    \begin{aligned}
        \nabla_h\cdot:\eh&\rightarrow\ch,\\
        \vec{f}&\mapsto\nabla_h\cdot \vec{f}=d_xf^x+d_yf^y.
    \end{aligned}
\end{equation}
Following from the discrete operators above, the 2D discrete Laplacian is naturally given by
\begin{equation}
    \begin{aligned}
        \Delta_h:\ch&\rightarrow\ch,\\
        \phi &\mapsto (\Delta_h \phi)_{i,j}\\
        &\quad =(\nabla_h\cdot \nabla_h \phi)_{i,j}\\
        &\quad =\frac{1}{h^2}(\phi_{i+1,j}+\phi_{i-1,j}+\phi_{i,j+1}+\phi_{i,j-1}-4\phi_{i,j}).
    \end{aligned}
\end{equation}

Over the domain $\Omega$ with periodic boundary conditions, we define the following grid inner products,
\begin{subequations}
    \begin{align}
        \langle \phi,\phi'\rangle_h&=h^2\sum_{i,j=1}^N \phi_{i,j}\phi'_{i,j},\qquad \phi,\phi'\in \ch,\\
        \langle \phi,\phi'\rangle_h&=h^2\sum_{i,j=1}^N \phi_{i+\frac{1}{2},j}\phi'_{i+\frac{1}{2},j},\qquad \phi,\phi'\in \ehx,\\
        \langle \phi,\phi'\rangle_h&=h^2\sum_{i,j=1}^N \phi_{i,j+\frac{1}{2}}\phi'_{i,j+\frac{1}{2}},\qquad \phi,\phi'\in \ehy,\\
        \langle \vec{f}_1,\vec{f}_2\rangle_h&=\langle f_1^x,f_2^x\rangle_h + \langle f_1^y,f_2^y\rangle_h,\qquad \vec{f_1},\vec{f_2}\in \eh.
    \end{align}
\end{subequations}
These discrete inner products induce, in the canonical way, the corresponding discrete norms.
\begin{subequations}
    \begin{align}
        \|\phi\|_h&=\sqrt{\langle \phi,\phi \rangle_h},\qquad \phi\in \ch,\\
        \|\vec{f}\|_h&=\sqrt{\langle \vec{f},\vec{f} \rangle_h},\qquad \vec{f}\in \eh.
    \end{align}
\end{subequations}

\begin{proposition}[Summation by parts formula]
\label{prop:sbp}
    For grid functions $\phi \in\ch,\vec{f}\in \eh$, the following summation by parts formula is valid:
    \begin{equation}\label{eq:sbp}
        \langle \phi,\nabla_h \cdot \vec{f}\rangle_h=
        -\langle \nabla_h\phi, \vec{f}\rangle_h.
    \end{equation}
\end{proposition}
\begin{proof}
    By direct computation, we have
\begin{equation}
    \begin{aligned}
        &\langle \phi,\nabla_h \cdot \vec{f}\rangle_h \\    
        &=h^2\sum_{i,j=1}^N  \phi_{i,j}\frac{1}{h}(f^x_{i+\frac{1}{2},j}-f^x_{i-\frac{1}{2},j}
        )+h^2\sum_{i,j=1}^N \phi_{i,j}\frac{1}{h}(f^y_{i,j+\frac{1}{2}}-f^y_{i,j-\frac{1}{2}}
        )\\
        &=h^2\sum_{i,j=1}^N \frac{1}{h}(\phi_{i,j}-\phi_{i+1,j})f^x_{i+\frac{1}{2},j}+h^2\sum_{i,j=1}^N \frac{1}{h}(\phi_{i,j}-\phi_{i,j+1})f^y_{i,j+\frac{1}{2}}\\
        &=-h^2\sum_{i,j=1}^N (D_x \phi)_{i+\frac{1}{2},j}
        f^x_{i+\frac{1}{2},j}
        -h^2\sum_{i,j=1}^N (D_y \phi)_{i,j+\frac{1}{2}}
        f^y_{i,j+\frac{1}{2}}\\
        &=-\langle \nabla_h \phi, \vec{f} \rangle_h.
    \end{aligned}
\end{equation}
    
\end{proof}
\begin{corollary}
    For grid functions $\phi,\psi\in\ch$, the following summation by parts formula is valid \cite{morton1994numerical}:
    \begin{equation}\label{eq:sbp2}
        \langle \phi,\Delta_h \psi\rangle_h=
        -\langle \nabla_h\phi, \nabla_h \psi \rangle_h.
    \end{equation}
\end{corollary}
\begin{proof}
    Take $\vec{f}=\nabla_h \psi$ in \eqref{eq:sbp}.
\end{proof}
\begin{corollary}
    The discrete operator $-\Delta_h$ is positive semi-definite on $\ch$ \cite{morton1994numerical}.
\end{corollary}
\begin{proof}
    Taking $\psi=\phi$ in \eqref{eq:sbp2}, we obtain that for any grid function $\phi \in \ch$,
    \begin{equation}
        \langle \phi,-\Delta_h \phi\rangle_h=
        \|\nabla_h \phi\|_h^2\geq0.
    \end{equation}
    Equality holds if and only if $\phi$ is constant on $\ch$.
\end{proof}
\section{The Iterative Solver Based on ADMM Framework}
\label{section:admm}
In this section, we present the methodology for solving the nonlinear system arising at each time step from the convex splitting scheme. 
We first detail the full derivation and algorithm of our proposed iterative solver based on ADMM framework. 


Let $T>0$ be the final computational time. We consider a uniform partition of the time interval $[0, T]$ with a time step size $\tau=\frac{T}{M}$ for some $M\in \mathbb{N}_+$, yielding discrete time levels $0=t_0<t_1<\cdots<t_M=T$. The numerical solution at time $t_n=n\tau$ is denoted by $u^n$.


The energy functional admits a decomposition into the sum of a convex and a concave part,
\begin{equation}
    \begin{aligned}
    E(u)&=\int_{\Omega}\left(\frac{\epsilon^2}{2}|\nabla u|^2+u\log(u)+(1-u)\log(1-u)\right)\text{d} x\text{d} y\\
    &\quad +\int_{\Omega}\theta (u-u^2)\text{d} x\text{d} y.
    \end{aligned}
\end{equation}

The temporal convex splitting scheme is constructed by treating the convex part implicitly and the concave part explicitly \cite{eyre1998unconditionally}.
Coupled with the spatial discretization introduced in \Cref{section:discrete}, the fully discrete scheme is stated as follows:

For a given $u^n\in \ch$, find $u^{n+1}\in \ch$ such that
\begin{equation}
\label{scheme:cs}
    \frac{u^{n+1}-u^n}{\tau}-\epsilon^2\Delta_hu^{n+1}
    +\log(u^{n+1})-\log(1-u^{n+1})+\theta(1-2u^n)=0,
\end{equation}
which can also be written as\begin{equation}
\label{eq:cs}
    \left(\frac{1}{\tau}-\epsilon^2\Delta_h\right)u^{n+1}+\log(u^{n+1})-\log(1-u^{n+1})=\frac{1}{\tau}u^{n}-\theta(1-2u^n).
\end{equation}

The remaining problem is how to solve this nonlinear system efficiently. In the following, we present our algorithmic approach.

Motivated by the need for robust solvers, our main strategy is to reformulate the numerical scheme \eqref{scheme:cs}
into a convex optimization problem \cite{boyd2004convex}, thereby enabling us to leverage powerful optimization tools.

For a given $u^n\in \ch$, we define a discrete energy functional
\begin{equation}
    \begin{aligned}
    J(u)&=\frac{1}{2\tau}\|u-u^n\|_h^2+\frac{\epsilon^2}{2}\|\nabla_hu\|_h^2\\
    &\quad +\langle u,\log(u)\rangle_h+
    \langle 1-u,\log(1-u)\rangle_h
    +\theta \langle u,1-2u^n \rangle_h.
    \end{aligned}
\end{equation}
Note that the presence of the logarithmic terms implicitly restricts the effective domain of $J(u)$ to the set $\{u \in \ch \mid 0 < u_{i,j} < 1, \forall i,j \}$.
The numerical scheme \eqref{scheme:cs} is then equivalent to the following convex optimization problem \cite{du2020phase}:
\begin{equation}\label{eq:opt}
    u^{n+1}=\arg \min_{u\in \ch}J(u).
\end{equation}
We then decompose the energy functional $J(u)$ into a sum of two convex terms, that is,
\begin{subequations}\label{eq:Ju}
    \begin{align}
        &J(u)=J_1(u)+J_2(u),\\ 
        &J_1(u)=\frac{1}{2\tau}\|u-u^n\|_h^2+\frac{\epsilon^2}{2}\|\nabla_hu\|_h^2,\\
        &J_2(u)=\langle u,\log(u)\rangle_h+
        \langle 1-u,\log(1-u)\rangle_h
        +\theta \langle u,1-2u^n \rangle_h,
    \end{align}
\end{subequations}
and now formalize the optimization problem by putting it into the following standard formulation, which reveals its structure clearly.
\begin{mini!}
{u_1, u_2}{J_1(u_1)+J_2(u_2),}{}{}
\addConstraint{u_1}{=u_2}{}
\addConstraint{u_i}{\in \ch,}{\quad i=1,2.}
\end{mini!}

Accordingly, the optimization problem is posed in terms of its Lagrange function, defined by
\begin{equation}\label{eq:Lagrange}
    \mathcal{L}(u_1,u_2;u_3)=J_1(u_1)+J_2(u_2)+
    \langle u_3,u_1-u_2 \rangle_h,
\end{equation}
where $u_3 \in \ch$ is the Lagrange multiplier \cite{rockafellar1973multiplier}, and the corresponding augmented Lagrange function is
\begin{equation}
    \mathcal{L}_{\rho}(u_1,u_2;u_3)=\mathcal{L}(u_1,u_2;u_3)+\frac{\rho}{2}\|u_1-u_2\|_h^2,
\end{equation}
where $\rho>0$ is the penalty parameter. We apply the standard ADMM framework \cite{neal2011distributed} to the optimization problem, whose $k$-th iteration step is given by
\begin{subequations}
    \begin{align}
        \label{eq:subp1}
        &u_2^{(k+1)}=\arg \min_{u_2\in \ch} \mathcal{L}_{\rho}(u_1^{(k)},u_2;u_3^{(k)}),\\
        \label{eq:subp2}
        &u_1^{(k+1)}=\arg \min_{u_1\in \ch} \mathcal{L}_{\rho}(u_1,u_2^{(k+1)};u_3^{(k)}),\\
        &u_3^{(k+1)}=u_3^{(k)}+\alpha\rho(u_1^{(k+1)}-u_2^{(k+1)}),
    \end{align}
\end{subequations}
where $\alpha\in(0,\frac{\sqrt{5}+1}{2})$ is the step size for multiplier update.

The efficiency of our ADMM implementation hinges on solving the two subproblems in \eqref{eq:subp1} and \eqref{eq:subp2}. 
We solve them by exploiting their distinct mathematical structures.

The first-order optimality condition for the $u_2$-update is given  by
\begin{equation}
    \label{eq:u2}
    \log(u_2^{(k+1)})-\log(1-u_2^{(k+1)})+
    \theta(1-2u^n)-u_3^{(k)}-\rho(u_1^{(k)}-u_2^{(k+1)})=0.
\end{equation}
Specifically, the optimality condition \eqref{eq:u2} reduces to solving a set of independent, one-dimensional nonlinear equations at each grid point $\{(i,j) \mid 1\leq i,j \leq N\}$,
completely decoupling the global problem into $N^2$ independent scalar equations.
This remarkable structure allows us to solve all $N^2$ equations perfectly in parallel, utilizing a safeguarded one-dimensional Newton's method for each equation \cites{dennis1996numerical,ortega2000iterative}.


The $u_1$-update results in a symmetric positive-definite linear system, i.e., 
    \begin{equation}
    \label{eq:u1}
        \left(\frac{1}{\tau}-\epsilon^2\Delta_h+\rho\right)u_1^{(k+1)}=\frac{1}{\tau}u^n+\rho u_2^{(k+1)}-u_3^{(k)},
    \end{equation}
which we solve with a preconditioned conjugate gradient (PCG) method \cite{concus1976generalized}. 
For uniform grids, this system is circulant and can be solved via discrete fast Fourier transform (FFT) techniques \cite{cooley1965algorithm}.

To summarize, we present the detailed process of the iterative solver based on ADMM framework in \Cref{alg:ADMM}.
\begin{algorithm}[!t]
    \caption{Iterative solver based on ADMM framework}
\label{alg:ADMM}
    \raggedright
    1. Define $u_1^{(0)}=u^n$, $u_2^{(0)}=u^n$, and
    $u_3^{(0)}=0$.\\
    2. Select $\rho\in(0,+\infty), \alpha \in (0,\frac{\sqrt{5}+1}{2})$.\\
    3. For $k\in\mathbb{N}$, update $u_2$ by solving the nonlinear system
    \begin{equation}
    \log(u_2^{(k+1)})-\log(1-u_2^{(k+1)})+
    \theta(1-2u^n)-u_3^{(k)}-\rho(u_1^{(k)}-u_2^{(k+1)})=0,
    \end{equation}
    \hspace{10pt} for $u_2^{(k+1)}\in\ch$ with one-dimensional Newton's method at each grid point.\\
    4. For $k\in\mathbb{N}$, update $u_1$ by solving the linear system
    \begin{equation}
    \label{u1}
        \left(\frac{1}{\tau}-\epsilon^2\Delta_h+\rho\right)u_1^{(k+1)}=\frac{1}{\tau}u^n+\rho u_2^{(k+1)}-u_3^{(k)},
    \end{equation}
    \hspace{10pt} for $u_1^{(k+1)}\in\ch$ with PCG or FFT.\\ 
    5. Update the Lagrange multiplier $u_3^{(k+1)}\in\ch$ as
    \begin{equation}\label{eq:u3}
    u_3^{(k+1)}=u_3^{(k)}+\alpha \rho(u_1^{(k+1)}-u_2^{(k+1)}).
    \end{equation}
\end{algorithm}

\begin{remark}
We explain why the one-dimensional Newton's method used in the $u_2$-update is safe and effective. 
Each subproblem in the $u_2$-update is equivalent to finding the unique root $x^*\in (0,1)$ of the function 
$h(x)=\log(x)-\log(1-x)+\rho x-m$ for some $m\in \mathbb{R}$. 
Note that $h(x)$ is strictly increasing on $(0,1)$, concave on $(0,\frac{1}{2}]$, and convex on $[\frac{1}{2},1)$. Furthermore, $\lim_{x\to 0^+}h(x)=-\infty$ and $\lim_{x\to 1^-}h(x)=+\infty$. 
If $h(\frac{1}{2})<0$, then upon selecting an initial guess $x^0$ such that $h(x^0)>0$,
the convexity of $h(x)$ on $[\frac{1}{2},1)$ rigorously guarantees that the sequence $\{x^k\}$
generated by Newton's method is strictly decreasing and unconditionally converges to the unique root $x^*\in(\frac{1}{2},x^0)$.
The case where $h(\frac{1}{2}) \ge 0$ is completely analogous.
Consequently, this simple initialization strategy naturally endows Newton's method with a robust safeguard. Furthermore, the condition $h'(x^*)>0$ ensures local quadratic convergence.


\end{remark}

\section{Theoretical Analysis}
\label{section:analysis}
In this section, we establish the following rigorous theoretical results for the proposed numerical method: 
(a) well-posedness and discrete energy stability to the convex splitting scheme;
(b) unconditional linear convergence and bound-preserving property for the ADMM solver.
\subsection{Well-posedness and discrete energy stability}
We first prove the existence and uniqueness of the discrete solution of the convex splitting scheme 
\eqref{scheme:cs} at each time step.
\begin{theorem}\label{thm:well-posedness}
    Assume that $0<u^n<1$ and $u^n \in \ch$. For any time step size $\tau>0$, there exists a unique $u^{n+1}\in \ch$ to solve \eqref{scheme:cs}.
\end{theorem}
\begin{proof}
    From \eqref{eq:opt}, the proof reduces to establishing the existence and uniqueness of a minimizer. 
    The second derivative of the objective function is given by
    \begin{equation}
        \nabla^2 J(u)=\frac{1}{\tau}I-\epsilon^2\Delta_h+G,
    \end{equation}
    where $G=\text{diag}\left \{\frac{1}{u_{i,j}^n(1-u_{i,j}^n)}\right \}_{i,j=1}^{N}$. 

    Here and hereafter, for a discrete functional $J:C_h\to\mathbb{R}$,
    $\nabla J(u)$ denotes the first derivative with respect to $u$, identified with
    its Riesz representative in $C_h$ through the discrete inner product
    $\langle\cdot,\cdot\rangle_h$. The second derivative $\nabla^2 J(u)$ is identified
    with the corresponding self-adjoint linear operator on $C_h$.

    Noticing the objective function $J(u)$ is $\frac{1}{\tau}$-strongly convex because of the property that $\nabla^2J(u)-\frac{1}{\tau}I$ is always positive-definite, 
    it suffices to show that a minimizer exists; uniqueness is then a direct consequence of strong convexity \cite{boyd2004convex}.
    
    Existence. If one component $u_{i,j}\to 0^+$ or $u_{i,j} \to 1^-$, the logarithmic term $\log(u_{i,j})-\log(1-u_{i.j})$ tends to $-\infty$ or $+\infty$ while other terms are all finite. 
    Therefore, the $\{i,j\}$-th equation in \eqref{scheme:cs} cannot hold under such a limit. 
    We thus conclude the existence of a positive constant $\delta>0$ such that $u_{i,j}\in[\delta,1-\delta]$ for all $1\leq i,j\leq N$. 
    A direct application of the extreme value theorem to the continuous function $J(u)$ on the compact constraint set $[\delta,1-\delta]^{N^2}$ 
    readily yields the existence of at least one minimizer \cites{1976Principles,weierstrass1894mathematische}.
\end{proof}

For $u\in \ch$, the discrete energy is defined by
\begin{equation}
\label{eq:discrete-energy}
    E_h(u)=\langle F(u),1 \rangle_h+\frac{\epsilon^2}{2}\|\nabla_h u\|_h^2,
\end{equation}
where $F(u)=u\log(u)+(1-u)\log(1-u)+\theta(u-u^2)$. Analogous to the monotonic decrease of the continuous energy, we prove a crucial discrete counterpart: 
the fully discrete energy obtained from the convex splitting solution is non-increasing, i.e., it exhibits unconditional discrete energy stability.
\begin{theorem}\label{thm:energy-stability} 
    Assume that $0< u^0<1$. For any time step size $\tau>0$, the discrete energy of convex splitting scheme \eqref{scheme:cs} is dissipated with time steps, i.e.
    \begin{equation}
        E_h(u^{n+1})\leq E_h(u^n).
    \end{equation}
\end{theorem}
\begin{proof}
    The proof proceeds by adapting the convexity arguments established in \cite{eyre1998unconditionally}. 
    We decompose $F(u)$ into $F(u)=F_1(u)-F_2(u)$, where $F_1(u)=u\log(u)+(1-u)\log(1-u)$ and  $F_2(u)=\theta(u^2-u)$ are both convex. 
    From \eqref{eq:opt}, $J(u^{n+1})\leq J(u^n)$, and a simple fact that $\langle v,w \rangle_h = \langle vw,1 \rangle_h$ for $v,w \in \ch$, we have
    \begin{equation}
        \begin{aligned}
        &\frac{1}{2\tau}\|u^{n+1}-u^n\|_h^2+\left(\frac{\epsilon^2}{2}\|\nabla_hu^{n+1}\|_h^2+\langle F_1(u^{n+1}),1 \rangle_h\right)
        \\
        \leq& \langle u^{n+1}-u^n, F_2'(u^n)\rangle_h
        +\left(\frac{\epsilon^2}{2}\|\nabla_hu^{n}\|_h^2+\langle F_1(u^{n}),1 \rangle_h\right)
        ,
    \end{aligned}
    \end{equation}
the convexity of $F_2(u)$ implies\begin{equation}
    \langle u^{n+1}-u^n, F_2'(u^n)\rangle_h\leq 
    \langle F_2(u^{n+1})-F_2(u^n),1 \rangle_h.
\end{equation}
    After simplification, we arrive at
\begin{equation}
    E_h(u^{n+1})+\frac{1}{2\tau}\|u^{n+1}-u^n\|_h^2\leq E_h(u^n).
\end{equation}
\end{proof}

\subsection{Unconditional linear convergence}
We begin by establishing the unconditional convergence of the algorithm, after which we prove that it converges at a linear rate.

The following lemma states the existence and uniqueness of the stationary point of the Lagrange function $\mathcal{L}$. 
\begin{lemma}
    Assume that $0<u^n<1$ and $u^n \in \ch$. For any time step size $\tau>0$, there exists a unique stationary point $(u_1^*,u_2^*,u_3^*)\in \ch \times \ch \times \ch$ of the Lagrange function $\mathcal{L}$ defined in \eqref{eq:Lagrange}.
\end{lemma}
\begin{proof}
    The stationary points are found by setting the partial derivatives of the function with respect to $u_1,u_2,u_3$ to zero, 
    which gives the following KKT system,
    \begin{subequations}\label{eq:stationary}
        \begin{align}
        &\frac{\partial \mathcal{L}}{\partial u_1}(u_1^*,u_2^*;u_3^*)
        =\frac{1}{\tau}(u_1^*-u^n)-\epsilon^2\Delta_hu_1^*+u_3^*=0,\\
        &\frac{\partial \mathcal{L}}{\partial u_2}(u_1^*,u_2^*;u_3^*)
        =\log(u_2^*)-\log(1-u_2^*)+\theta(1-2u^n)-u_3^*=0,\\
        &\frac{\partial \mathcal{L}}{\partial u_3}
        (u_1^*,u_2^*;u_3^*)
        =u_1^*-u_2^*=0.
        \end{align}
    \end{subequations}
    The above KKT system reduces to
    \begin{equation}
        \frac{u_1^*-u^n}{\tau}-\epsilon^2\Delta_hu_1^*
    +\log(u_1^*)-\log(1-u_1^*)+\theta(1-2u^n)=0,
    \end{equation}
    which is exactly obtained by substituting $u_1^*$ for $u^{n+1}$ in \eqref{scheme:cs}. By \ref{thm:well-posedness}, we arrive at $u_1^*=u^{n+1}$, uniquely. 
    Define $(u_1^*,u_2^*,u_3^*)\in \ch \times \ch \times \ch$ as
    \begin{subequations}
        \begin{align}
            &u_1^*=u^{n+1},\\
            &u_2^*=u^{n+1},\\
            &\begin{aligned}
                u_3^*&=-\frac{1}{\tau}(u_1^{n+1}-u^n)+\epsilon^2\Delta_hu_1^{n+1}\\
                &=\log(u_2^{n+1})-\log(1-u_2^{n+1})+\theta(1-2u^n).
            \end{aligned}
        \end{align}
    \end{subequations}
    Then $(u_1^*,u_2^*,u_3^*)$ satisfies the system \eqref{eq:stationary}, which means it is the unique stationary point of the Lagrange function $\mathcal{L}$.
\end{proof}
To prove the convergence of \Cref{alg:ADMM}, the sequence generated by the iterative solver is denoted by $(u_1^{(k)},u_2^{(k)},u_3^{(k)})$, and we adopt the notation $(u_1^*,u_2^*,u_3^*)$ established in the previous lemma. 
We define the error sequence
\begin{equation}
    (e_1^{(k)},e_2^{(k)},e_3^{(k)})=(u_1^{(k)}-u_1^*,u_2^{(k)}-u_2^*,u_3^{(k)}-u_3^*),\qquad k\in \mathbb{N},
\end{equation}
and the sequence $\{\Phi^{(k)}\}$ as
\begin{equation}\label{def:Phi}
    \Phi^{(k)}=\frac{1}{\alpha\rho}\|e_3^{(k)}\|_h^2+\rho\|e_1^{(k)}\|_h^2+
    \eta\rho\|e_1^{(k)}-e_2 ^{(k)}\|_h^2,\qquad k\in \mathbb{N},
\end{equation}
where $\eta=\max(1-\alpha,1-\alpha^{-1})$. The techniques of convergence analysis are adapted from \cites{neal2011distributed,chen2017note,li2025overcoming}. 
We introduce the following useful inequality.
\begin{lemma}\label{lemma:Phi}
With the notations above, the following inequality holds,
\begin{equation}
\begin{aligned}
    &\Phi^{(k)}-\Phi^{(k+1)}\\
    &\geq \alpha c(\alpha)\rho\|e_1^{(k+1)}-e_1^{(k)}\|_h^2
    +c(\alpha)\rho\|e_1^{(k+1)}-e_2^{(k+1)}\|_h^2,
\end{aligned}
\end{equation}
for all $k \in \mathbb{N}$, where $c(\alpha)=\min(1,1+\alpha^{-1}-\alpha)=\begin{cases}
    1,&\alpha \in (0,1]\\
    1+\alpha^{-1}-\alpha,&\alpha \in (1,+\infty)
\end{cases}$.
\end{lemma}
\begin{proof}
From \eqref{def:Phi}, we have
\begin{equation}
    \begin{aligned}
    &\Phi^{(k)}-\Phi^{(k+1)}\\
    &=\frac{1}{\alpha \rho}\left(\|e_3^{(k)}\|_h^2-\|e_3^{(k+1)}\|_h^2\right)+\rho \left(\|e_1^{(k)}\|_h^2-\|e_1^{(k+1)}\|_h^2\right)\\
    &\quad +\eta \rho \left(\|e_1^{(k)}-
    e_2^{(k)}\|_h^2-\|e_1^{(k+1)}-
    e_2^{(k+1)}\|_h^2\right)\\
    &=-2\langle e_3^{(k)},e_1^{(k+1)}-e_2^{(k+1)}\rangle_h-\alpha \rho \|e_1^{(k+1)}-e_2^{(k+1)}\|_h^2\\
    &\quad +\rho \left(\|e_1^{(k+1)}-e_1^{(k)}\|_h^2
    +2\langle e_1^{(k)}-e_1^{(k+1)},e_1^{(k+1)}\rangle_h\right)\\
    &\quad +\eta \rho \left(\|e_1^{(k)}-
    e_2^{(k)}\|_h^2-\|e_1^{(k+1)}-
    e_2^{(k+1)}\|_h^2\right).
\end{aligned}
\end{equation}
From \eqref{eq:u1}, \eqref{eq:u2}, \eqref{eq:u3}, and the KKT system \eqref{eq:stationary}, we get the error equations,
\begin{subequations}
    \begin{align}
        &g(u_2^{(k+1)})-g(u_2^*)-\rho(e_1^{(k)}-e_2^{(k+1)})-e_3^{(k)}=0,
        \label{eq:e2}
        \\
        &(\tau^{-1}-\epsilon^2 \Delta_h)e_1^{(k+1)}
        +\rho(e_1^{(k+1)}-e_2^{(k+1)})+e_3^{(k)}=0,
        \label{eq:e1}
        \\
        &e_3^{(k+1)}=e_3^{(k)}+\alpha \rho(e_1^{(k+1)}-e_2^{(k+1)}),
        \label{eq:e3}
    \end{align}
\end{subequations}
where $g(u)=\log(u)-\log(1-u)$. Equation \eqref{eq:e1} implies
\begin{equation}
    \langle (\tau^{-1}-\epsilon^2\Delta_h)e_1^{(k+1)}
        +\rho(e_1^{(k+1)}-e_2^{(k+1)})+e_3^{(k)}, e_1^{(k+1)} \rangle_h=0.
\end{equation}
The positive definiteness of the operator $(\tau^{-1}-\epsilon^2\Delta_h)$ leads to
\begin{equation}\label{eq:inequality1}
    \langle e_3^{(k)},e_1^{(k+1)}\rangle_h 
    \leq
    -\rho \langle e_1^{(k+1)}-e_2^{(k+1)},e_1^{(k+1)}\rangle_h.
\end{equation}
\eqref{eq:e2} implies
\begin{equation}
    \langle g(u_2^{(k+1)})-g(u_2^*)-\rho(e_1^{(k)}-e_2^{(k+1)})-e_3^{(k)},u_2^{(k+1)}-u_2^* \rangle_h=0.
\end{equation}
The monotonicity property of $g(u)$ yields
\begin{equation}
    \langle g(u_2^{(k+1)})-g(u_2^*),e_2^{(k+1)} \rangle_h=
    \langle \rho(e_1^{(k)}-e_2^{(k+1)})+e_3^{(k)},e_2^{(k+1)} \rangle_h\geq 0,
\end{equation}
which reduces to 
\begin{equation}\label{eq:inequality2}
    \langle e_3^{(k)},e_2^{(k+1)} \rangle _h\geq -\rho 
    \langle e_1^{(k)}-e_2^{(k+1)},e_2^{(k+1)} \rangle_h.
\end{equation}
It follows from \eqref{eq:inequality1} and \eqref{eq:inequality2} that
\begin{equation}
    -2\langle e_3^{(k)},e_1^{(k+1)}-e_2^{(k+1)}\rangle_h\geq 2\rho \langle e_1^{(k+1)}-e_2^{(k+1)},e_1^{(k+1)}\rangle_h - 2\rho
    \langle e_1^{(k)}-e_2^{(k+1)},e_2^{(k+1)} \rangle_h,
\end{equation}
which leads to
\begin{equation}
\begin{aligned}
&\Phi^{(k)}-\Phi^{(k+1)}\\
&\geq 2\rho \langle e_1^{(k+1)}-e_2^{(k+1)},e_1^{(k+1)}\rangle_h - 2\rho
    \langle e_1^{(k)}-e_2^{(k+1)},e_2^{(k+1)} \rangle_h \\
    &\quad -\alpha \rho \|e_1^{(k+1)}-e_2^{(k+1)}\|_h^2\\
    &\quad +\rho \left(\|e_1^{(k+1)}-e_1^{(k)}\|_h^2
    +2\langle e_1^{(k)}-e_1^{(k+1)},e_1^{(k+1)}\rangle_h \right)\\
    &\quad +\eta \rho \left(\|e_1^{(k)}-
    e_2^{(k)}\|_h^2-\|e_1^{(k+1)}-
    e_2^{(k+1)}\|_h^2 \right)\\
&=(2-\alpha-\eta)\rho \|e_1^{(k+1)}-e_2^{(k+1)}\|_h^2 \\
&\quad +2\rho \langle e_1^{(k+1)}-e_2^{(k+1)},e_1^{(k)}-e_1^{(k+1)} \rangle_h\\
&\quad +\rho \|e_1^{(k+1)}-e_1^{(k)}\|_h^2+\eta \rho \|e_1^{(k)}-e_2^{(k)}\|_h^2.
\end{aligned}
\end{equation}
Taking $k$-th and $k+1$-th in \eqref{eq:e1}, we obtain that
\begin{equation}
    \begin{aligned}
    &(\tau^{-1}-\epsilon^2 \Delta_h)(e_1^{(k+1)}-e_1^{(k)})+\rho (e_1^{(k+1)}-e_2^{(k+1)})\\
    & -\rho(e_1^{(k)}-e_2^{(k)})+(e_3^{(k)}-e_3^{(k-1)})=0.
    \end{aligned}
\end{equation}
We take the inner product with $(e_1^{(k+1)}-e_1^{(k)})$. Then, by exploiting the positive-definiteness of $(\tau^{-1}-\epsilon^2 \Delta_h)$ 
and the definition of $e_3^{(k)}$, we infer the following useful inequality that
\begin{equation}\label{eq:inequality4}
    \langle e_1^{(k+1)}-e_2^{(k+1)},e_1^{(k+1)}-e_1^{(k)}\rangle_h
    \leq(1-\alpha) \langle e_1^{(k)}-e_2^{(k)},
    e_1^{(k+1)}-e_1^{(k)}\rangle_h,
\end{equation}
which implies
\begin{equation}
\begin{aligned}
    &\Phi^{(k)}-\Phi^{(k+1)}\\
    &\geq (2-\alpha-\eta)\rho \|e_1^{(k+1)}-e_2^{(k+1)}\|_h^2\\
    &\quad -2(1-\alpha)\rho \langle e_1^{(k)}-e_2^{(k)}, e_1^{(k+1)}-e_1^{(k)}\rangle_h\\
    &\quad +\rho \|e_1^{(k+1)}-e_1^{(k)}\|_h^2+\eta \rho \|e_1^{(k)}-e_2^{(k)}\|_h^2.
\end{aligned}
\end{equation}

If $\alpha \in (0,1]$, $\eta=1-\alpha$. From the inequality $2ab\leq a^2+b^2$, we derive that
\begin{equation}\label{eq:a<1}
    \begin{aligned}
    &\Phi^{(k)}-\Phi^{(k+1)}\\
    &\geq (2-\alpha-(1-\alpha))\rho \|e_1^{(k+1)}-e_2^{(k+1)}\|_h^2\\
    &\quad -(1-\alpha)\rho (\| e_1^{(k)}-e_2^{(k)}\|_h^2+
    \|e_1^{(k+1)}-e_1^{(k)}\|_h^2)\\
    &\quad +\rho \|e_1^{(k+1)}-e_1^{(k)}\|_h^2+(1-\alpha) \rho \|e_1^{(k)}-e_2^{(k)}\|_h^2\\
    &=\alpha \rho \|e_1^{(k+1)}-e_1^{(k)}\|_h^2+\rho \|e_1^{(k+1)}-e_2^{(k+1)}\|_h^2.
\end{aligned}
\end{equation}
If $\alpha \in (1,+\infty)$, $\eta=1-\alpha^{-1}$. It follows from the inequality $2ab\leq \alpha^{-1}a^2+\alpha b^2$ that
\begin{equation}\label{eq:a>1}
\begin{aligned}
&\Phi^{(k)}-\Phi^{(k+1)}\\
    &\geq (2-\alpha-(1-\alpha^{-1}))\rho \|e_1^{(k+1)}-
    e_2^{(k+1)}\|_h^2\\
    &\quad +(\alpha-1)\rho (-\alpha^{-1}\| e_1^{(k)}-e_2^{(k)}\|_h^2-
    \alpha \|e_1^{(k+1)}-e_1^{(k)}\|_h^2)\\
    &\quad +\rho \|e_1^{(k+1)}-e_1^{(k)}\|_h^2
    +(1-\alpha^{-1}) \rho \|e_1^{(k)}-
    e_2^{(k)}\|_h^2\\
    &=(1+\alpha-\alpha^2)\rho \|e_1^{(k+1)}-e_1^{(k)}\|_h^2\\
    &\quad +(1-\alpha+\alpha^{-1})\rho \|e_1^{(k+1)}- e_2^{(k+1)}\|_h^2.
\end{aligned}
\end{equation}
The proof is completed by the above two inequalities \eqref{eq:a<1} and \eqref{eq:a>1}.
\end{proof}
Armed with the Lemma \ref{lemma:Phi}, we proceed to prove unconditional convergence.
\begin{theorem}\label{thm:convergence}
    Assume that the step size $\alpha \in (0,\frac{\sqrt{5}+1}{2})$ for multiplier update. Then the iterative solver \ref{alg:ADMM} converges,
    \begin{equation}
        \lim_{k\to +\infty} u_1^{(k)}=
        \lim_{k\to +\infty} u_2^{(k)}=u^{n+1}.
    \end{equation}
\end{theorem}
\begin{proof}
    Notice that $c(\alpha)>0$ is equivalent to $\alpha \in (0,\frac{\sqrt{5}+1}{2})$. By Lemma \ref{lemma:Phi}, the sequence $\Phi^{(k)}$ is bounded and convergent because it is monotonically decreasing and bounded below by zero \cite{1976Principles}, thus the sequences
    \begin{align*}
        \{e_3^{(k)}\},\{e_1^{(k)}\},\{e_1^{(k)}-e_2^{(k)}\}
    \end{align*}
    are bounded. Then, the sequences
       $ \{u_1^{(k)}\},\{u_2^{(k)}\},\{u_3^{(k)}\}$ are bounded. Therefore, there exists a subsequence $\{(u_1^{(k_l)},u_2^{(k_l)},u_3^{(k_l)}) \}$ which converges to the unique stationary point of $\mathcal{L}$, i.e.
       \begin{equation}
        \begin{aligned}
           \lim_{l\to +\infty}(u_1^{(k_l)},u_2^{(k_l)},u_3^{(k_l)})
           &=(u_1^*,u_2^*,u_3^*)\\
           &=\left(u^{n+1},u^{n+1},-\frac{1}{\tau}(u_1^{n+1}-u^n)+\epsilon^2\Delta_hu_1^{n+1}\right).
        \end{aligned}
       \end{equation}
      Since $\{\Phi^{(k)}\}$ is convergent and has a subsequence $\{\Phi^{(k_l)}\}$ converging to $0$, the sequence $\{\Phi^{(k)}\}$ converges to $0$ itself \cite{1976Principles}. Then we have
      \begin{equation}
          \lim_{k\to +\infty}\|e_3^{(k)}\|_h=\lim_{k\to +\infty}\|e_1^{(k)}\|_h=\lim_{k\to +\infty}\|e_1^{(k)}-e_2^{(k)}\|_h=0.
      \end{equation}
      The triangle inequality \begin{equation}
          \|e_2^{(k)}\|_h\leq \|e_2^{(k)}-e_1^{(k)}\|_h+\|e_1^{(k)}\|_h 
      \end{equation}
      implies 
      \begin{equation}
          \lim_{k\to +\infty}\|e_2^{(k)}\|_h=0. 
      \end{equation}
      We have thus completed the proof.
\end{proof}

In what follows, we prove that the algorithm converges at a linear rate. 
The proof relies on the following lemma, which provides a sufficient condition for linear convergence. 
\begin{lemma}\label{lm:J}
    The objective function $J(u)=J_1(u)+J_2(u)$ in \eqref{eq:Ju} satisfies the regularity conditions required for linear convergence: 
$J_1(u)$ is strongly convex with a Lipschitz continuous gradient, and $J_2(u)$ is strongly convex.
\end{lemma}
\begin{proof}
    By direct computation, one obtains
\begin{subequations}
\begin{align}
\label{eq:Hessian1}
&\nabla^2J_1(u)=\tau^{-1}I-\epsilon^2\Delta_h,\\
\label{eq:Hessian2}
&\nabla^2J_2(u)=G=\text{diag}\left \{\frac{1}{u_{i,j}(1-u_{i,j})}\right \}_{i,j=1}^{N}.
\end{align}
\end{subequations}
From \eqref{eq:Hessian1}, the function $J_1(u)$ is strongly convex with constant $\mu_1>0$ 
and has a Lipschitz continuous gradient with constant $L_1>0$, which are given by
\begin{equation}
\begin{aligned}
    \mu_1&=\tau^{-1},\\
    L_1&=\tau^{-1}+\frac{8\epsilon^2}{h^2}.
\end{aligned}
\end{equation}
This is due to the known spectral properties of the negative discrete Laplacian matrix $-\Delta_h$ 
with $\lambda_{\min}(-\Delta_h)=0$ and $\lambda_{\max}(-\Delta_h)=\frac{8}{h^2}$.

Similarly, with \eqref{eq:Hessian2}, the function $J_2(u)$ is strongly convex with constant $\mu_2=4>0$, 
yet its gradient is not Lipschitz continuous.
\end{proof}

Defining the sequence 
\begin{equation}
R^{(k)}=\rho \|e_1^{k}\|_h^2+\rho^{-1}\|e_3^{k}\|_h^2, 
\end{equation}
the following theorem describes the linear convergence rate of $\{R^{(k)}\}$ in \ref{alg:ADMM}.
\begin{theorem}
Let the sequence $\{R^{(k)}\}$ be defined as above. It is linearly convergent in the sense that there exists a constant $\delta>0$ such that 
\begin{equation}
    R^{(k)}\geq (1+\delta) R^{(k+1)},
\end{equation}
for all $k\in \mathbb{N}$. In the special case where the step size for multiplier update is fixed as $\alpha=1$, the constant $\delta$ takes the following explicit form:
\begin{equation}
\delta=2\left(\frac{\rho}{\mu_1}+\frac{L_1}{\rho}\right)^{-1}.
\end{equation}
Choosing $\rho=\sqrt{L_1\mu_1}$ yields the largest $\delta$,
\begin{equation}
\delta_{\max}=\sqrt{\frac{\mu_1}{L_1}}=\sqrt{\frac{h^2}{h^2+8\tau \epsilon^2}}.
\end{equation}
\end{theorem}
\begin{proof}
We prove the assertion for the case $\alpha=1$. Let
\begin{equation}
    A=\tau^{-1}I-\epsilon^2\Delta_h .
\end{equation}
Then $A$ is symmetric positive definite on $C_h$, and
\begin{equation}
    \mu_1\|v\|_h^2
    \le
    \langle Av,v\rangle_h
    \le
    L_1\|v\|_h^2,
    \qquad \forall v\in C_h.
\end{equation}

For $\alpha=1$, the multiplier update gives
\begin{equation}\label{eq:alpha1-multiplier}
    e_3^{(k+1)}
    =
    e_3^{(k)}
    +
    \rho\left(e_1^{(k+1)}-e_2^{(k+1)}\right).
\end{equation}
The $u_1$-update error equation is
\begin{equation}\label{eq:u1-error-alpha1}
    A e_1^{(k+1)}
    +
    \rho\left(e_1^{(k+1)}-e_2^{(k+1)}\right)
    +
    e_3^{(k)}
    =
    0 .
\end{equation}
Combining \eqref{eq:alpha1-multiplier} and \eqref{eq:u1-error-alpha1}, we obtain
\begin{equation}\label{eq:e3-Ae1}
    e_3^{(k+1)}
    =
    -A e_1^{(k+1)} .
\end{equation}

Next, the $u_2$-update error equation reads
\begin{equation}\label{eq:u2-error-alpha1}
    \nabla J_2(u_2^{(k+1)})
    -
    \nabla J_2(u_2^\ast)
    =
    e_3^{(k)}
    +
    \rho\left(e_1^{(k)}-e_2^{(k+1)}\right).
\end{equation}
Using \eqref{eq:alpha1-multiplier}, we have
\begin{equation}
    e_3^{(k)}
    =
    e_3^{(k+1)}
    -
    \rho\left(e_1^{(k+1)}-e_2^{(k+1)}\right).
\end{equation}
Substituting this identity into \eqref{eq:u2-error-alpha1} yields
\begin{equation}\label{eq:u2-error-new}
\begin{aligned}
    &\nabla J_2(u_2^{(k+1)})
    -
    \nabla J_2(u_2^\ast)\\
    &=
    e_3^{(k+1)}
    -
    \rho\left(e_1^{(k+1)}-e_2^{(k+1)}\right)  
    +
    \rho\left(e_1^{(k)}-e_2^{(k+1)}\right)  \\
    &=
    e_3^{(k+1)}
    -
    \rho\left(e_1^{(k+1)}-e_1^{(k)}\right).
\end{aligned}
\end{equation}
Since $J_2$ is convex, then by \eqref{eq:u2-error-new},
\begin{equation}\label{eq:J2-monotone-alpha1}
    \left\langle
    e_3^{(k+1)}
    -
    \rho\left(e_1^{(k+1)}-e_1^{(k)}\right),
    e_2^{(k+1)}
    \right\rangle_h
    \ge 0 .
\end{equation}
Using
\begin{equation}
    e_2^{(k+1)}
    =
    e_1^{(k+1)}
    -
    \left(e_1^{(k+1)}-e_2^{(k+1)}\right),
\end{equation}
we rewrite \eqref{eq:J2-monotone-alpha1} as
\begin{equation}
\begin{aligned}
0
&\le
\left\langle
e_3^{(k+1)}
-
\rho\left(e_1^{(k+1)}-e_1^{(k)}\right),
e_1^{(k+1)}
-
\left(e_1^{(k+1)}-e_2^{(k+1)}\right)
\right\rangle_h  \\
&=
\left\langle e_3^{(k+1)},e_1^{(k+1)}\right\rangle_h
-
\rho\left\langle
e_1^{(k+1)}-e_1^{(k)},
e_1^{(k+1)}
\right\rangle_h  \\
&\quad
-
\left\langle
e_3^{(k+1)},
e_1^{(k+1)}-e_2^{(k+1)}
\right\rangle_h  \\
&\quad
+
\rho\left\langle
e_1^{(k+1)}-e_1^{(k)},
e_1^{(k+1)}-e_2^{(k+1)}
\right\rangle_h .
\end{aligned}
\end{equation}
Consequently,
\begin{equation}\label{eq:key-inner-ineq}
\begin{aligned}
&\rho\left\langle
e_1^{(k+1)},
e_1^{(k+1)}-e_1^{(k)}
\right\rangle_h
+
\left\langle
e_3^{(k+1)},
e_1^{(k+1)}-e_2^{(k+1)}
\right\rangle_h  \\
&\le
\left\langle
e_1^{(k+1)},
e_3^{(k+1)}
\right\rangle_h
+
\rho\left\langle
e_1^{(k+1)}-e_1^{(k)},
e_1^{(k+1)}-e_2^{(k+1)}
\right\rangle_h .
\end{aligned}
\end{equation}

We now estimate $R^{(k)}-R^{(k+1)}$. 
By \eqref{eq:alpha1-multiplier} and
\begin{equation}
    e_1^{(k)}
    =
    e_1^{(k+1)}
    -
    \left(e_1^{(k+1)}-e_1^{(k)}\right),
\end{equation}
we have
\begin{equation}\label{eq:R-diff-expand}
\begin{aligned}
&R^{(k)}-R^{(k+1)}\\
&=
\rho\left(
\|e_1^{(k)}\|_h^2
-
\|e_1^{(k+1)}\|_h^2
\right)  
+
\rho^{-1}\left(
\|e_3^{(k)}\|_h^2
-
\|e_3^{(k+1)}\|_h^2
\right)  \\
&=
\rho\left\|
e_1^{(k+1)}-e_1^{(k)}
\right\|_h^2
-
2\rho
\left\langle
e_1^{(k+1)},
e_1^{(k+1)}-e_1^{(k)}
\right\rangle_h  \\
&\quad
+
\rho\left\|
e_1^{(k+1)}-e_2^{(k+1)}
\right\|_h^2
-
2
\left\langle
e_3^{(k+1)},
e_1^{(k+1)}-e_2^{(k+1)}
\right\rangle_h .
\end{aligned}
\end{equation}
Using \eqref{eq:key-inner-ineq} in \eqref{eq:R-diff-expand}, we obtain
\begin{equation}\label{eq:R-diff-lower1}
\begin{aligned}
&R^{(k)}-R^{(k+1)}\\
&\ge
\rho\left\|
e_1^{(k+1)}-e_1^{(k)}
\right\|_h^2
+
\rho\left\|
e_1^{(k+1)}-e_2^{(k+1)}
\right\|_h^2  
-
2\left\langle
e_1^{(k+1)},
e_3^{(k+1)}
\right\rangle_h  \\
&\quad
-
2\rho\left\langle
e_1^{(k+1)}-e_1^{(k)},
e_1^{(k+1)}-e_2^{(k+1)}
\right\rangle_h  \\
&=
-2\left\langle
e_1^{(k+1)},
e_3^{(k+1)}
\right\rangle_h  \\
&\quad
+
\rho
\left\|
\left(e_1^{(k+1)}-e_1^{(k)}\right)
-
\left(e_1^{(k+1)}-e_2^{(k+1)}\right)
\right\|_h^2  \\
&\ge
-2\left\langle
e_1^{(k+1)},
e_3^{(k+1)}
\right\rangle_h .
\end{aligned}
\end{equation}
By \eqref{eq:e3-Ae1}, this gives
\begin{equation}\label{eq:R-diff-lower-A}
    R^{(k)}-R^{(k+1)}
    \ge
    2\left\langle
    A e_1^{(k+1)},e_1^{(k+1)}
    \right\rangle_h .
\end{equation}

On the other hand, by \eqref{eq:e3-Ae1},
\begin{equation}\label{eq:R-kplus1-A}
\begin{aligned}
R^{(k+1)}
&=
\rho\|e_1^{(k+1)}\|_h^2
+
\rho^{-1}\|e_3^{(k+1)}\|_h^2  \\
&=
\rho\|e_1^{(k+1)}\|_h^2
+
\rho^{-1}\|A e_1^{(k+1)}\|_h^2 .
\end{aligned}
\end{equation}

We next prove that
\begin{equation}\label{eq:spectral-bound}
    2\langle Av,v\rangle_h
    \ge
    \delta\left(
    \rho\|v\|_h^2+\rho^{-1}\|Av\|_h^2
    \right),
    \qquad \forall v\in C_h,
\end{equation}
where
\begin{equation}\label{eq:delta-def}
    \delta
    =
    2\left(
    \frac{\rho}{\mu_1}
    +
    \frac{L_1}{\rho}
    \right)^{-1}
    =
    \frac{2\mu_1\rho}{\rho^2+\mu_1L_1}.
\end{equation}
Indeed, since
\begin{equation}
    \mu_1 I\preceq A\preceq L_1 I,
\end{equation}
we have
\begin{equation}
    A-\mu_1 I\succeq 0,
    \qquad
    L_1 I-A\succeq 0.
\end{equation}
With the fact that if both positive semidefinite operators are polynomials of the same symmetric operator $A$, then their
product is positive semidefinite, we have
\begin{equation}
    A(L_1-A)\succeq 0.
\end{equation}
Thus
\begin{equation}
    \rho^2(A-\mu_1 I)+\mu_1 A(L_1 I-A)\succeq 0.
\end{equation}
Expanding this inequality gives
\begin{equation}
    (\rho^2+\mu_1L_1)A-\mu_1\rho^2 I-\mu_1 A^2\succeq 0.
\end{equation}
Taking the inner product with any $v\in C_h$, we get
\begin{equation}
    (\rho^2+\mu_1L_1)\langle Av,v\rangle_h
    \ge
    \mu_1\rho^2\|v\|_h^2+\mu_1\|Av\|_h^2 .
\end{equation}
Multiplying both sides by $2/(\rho^2+\mu_1L_1)$ yields
\begin{equation}
    2\langle Av,v\rangle_h
    \ge
    \frac{2\mu_1\rho^2}{\rho^2+\mu_1L_1}\|v\|_h^2
    +
    \frac{2\mu_1}{\rho^2+\mu_1L_1}\|Av\|_h^2 .
\end{equation}
By the definition \eqref{eq:delta-def}, the right-hand side is precisely
\begin{equation}
    \delta\left(
    \rho\|v\|_h^2+\rho^{-1}\|Av\|_h^2
    \right).
\end{equation}
Therefore, \eqref{eq:spectral-bound} holds.

Taking $v=e_1^{(k+1)}$ in \eqref{eq:spectral-bound}, and using
\eqref{eq:R-diff-lower-A} and \eqref{eq:R-kplus1-A}, we obtain
\begin{equation}
    R^{(k)}-R^{(k+1)}
    \ge
    \delta R^{(k+1)}.
\end{equation}
Equivalently,
\begin{equation}
    R^{(k)}
    \ge
    (1+\delta)R^{(k+1)}.
\end{equation}
This proves the linear convergence of $\{R^{(k)}\}$.

Finally, since
\begin{equation}
    \delta
    =
    2\left(\frac{\rho}{\mu_1}+\frac{L_1}{\rho}\right)^{-1},
\end{equation}
the above lower bound for $\delta$ is maximized when
\begin{equation}
    \rho=\sqrt{\mu_1L_1}
\end{equation}
by the arithmetic-geometric mean inequality.
Substituting this value gives
\begin{equation}
    \delta_{\max}
    =
    \sqrt{\frac{\mu_1}{L_1}}.
\end{equation}
Since
\begin{equation}
    \mu_1=\tau^{-1},
    \qquad
    L_1=\tau^{-1}+\frac{8\epsilon^2}{h^2},
\end{equation}
we finally obtain
\begin{equation}
    \delta_{\max}
    =
    \sqrt{
    \frac{\tau^{-1}}
    {\tau^{-1}+8\epsilon^2/h^2}
    }
    =
    \sqrt{
    \frac{h^2}{h^2+8\tau\epsilon^2}
    }.
\end{equation}
The proof is complete.

\end{proof}

\begin{remark}
We briefly discuss the bound preserving property of the algorithm. 
For the update of $u_2$, the presence of the logarithmic term ensures $u_2^{(k)}\in(0,1)$ at a point-wise level for all $k \in \mathbb{N}$. 
\ref{thm:convergence} guarantees that $u_2^{(k)}$ is a good approximation of $u^{n+1}$ after a sufficiently large number of iterations. 
Consequently, we obtain the bound preserving property.
\end{remark}

\section{Numerical results}
\label{section:numerical}
This section presents numerical experiments designed to validate the theoretical analysis and evaluate the practical performance of the proposed algorithm. 
Specifically, we verify its core properties including linear convergence, bound-preserving property, and energy stability in \Cref{subsec:test}. 
Numerical simulations in two- and three-dimensional settings are demonstrated in \Cref{subsec:2d} and \Cref{subsec:3d}, respectively, 
illustrating the effectiveness and robustness of our solver.

\subsection{Convergence test}
\label{subsec:test}
In this numerical experiment, the parameters are set as follows:
\begin{equation}
    L=2\pi,\ \epsilon=0.10,\ \tau=10^{-4},\ T=0.1,\ \theta=4.0.
\end{equation}
The initial condition is
\begin{equation}
    u_0(x,y)=0.5+0.25\sin x\sin y.
\end{equation}

For the iterative solver, the penalty parameter $\rho$ is dynamically updated based on the primal and dual residuals to ensure balanced convergence. 
Specifically, the following dynamic update strategy is adopted.
\begin{equation}
\rho^{(k+1)}=
\begin{cases}
\gamma_1\rho^{(k)},&\text{if } r^{(k)}> \nu s^{(k)},\\
\rho^{(k)}/\gamma_2,&\text{if } s^{(k)}> \nu r^{(k)},\\
\rho^{(k)},&\text{otherwise},
\end{cases}
\end{equation}
where $r^{(k)}=\|u_1^{(k)}-u_2^{(k)}\|_h$ and $s^{(k)}=\|u_1^{(k)}-u_1^{(k-1)}\|_h$ denote the primal residual and the dual residual, respectively.
In our simulations, the adjustment parameters are chosen as $\nu=10$ and $\gamma_1=\gamma_2=2$, which yield well-balanced primal and dual residuals.

The convergence criterion is set as
\begin{equation}
    \max\{r^{(k)},s^{(k)}\}\leq \gamma=10^{-8},
\end{equation}

The numerical solution obtained with $N=512$, $\tau=10^{-4}$ is taken as the exact solution, and the spatial convergence rate is tested for $N=16,32,64,128,256$.
The temporal convergence order test is analogous.
The test results, displayed in \Cref{ta:rate}, confirm the second order in space and first order in time for our solver.

Fixing $N=256, \tau = 0.01$ and $T=1.0$, we show the linear convergence and other properties of the algorithm in \Cref{fig:1}.
From \Cref{fig:11}, we see the number of ADMM iterations at each time step, which demonstrates the high efficiency of the proposed algorithm. 
The discrete energy, plotted in \Cref{fig:12}, decays monotonically as expected, confirming energy stability.
\Cref{fig:13} and \Cref{fig:14} trace the maximum and minimum of $u$ over time, both staying safely within the physical bounds $(0,1)$, demonstrating the bound-preserving property. 
\Cref{fig:15} displays the residual during the second time step, 
where the residual is defined by
\begin{equation}
    R(u)=\left \|\frac{u-u^n}{\tau}-\epsilon^2\Delta_hu
    +\log(u)-\log(1-u)+\theta(1-2u^n) \right \|_h.
\end{equation}
The linear decay of the residual, as observed on a semi-logarithmic scale, evidently confirms the linear convergence rate of our solver.
Finally, \Cref{fig:16} captures the dynamical evolution process of the phase variable at different time steps.
\begin{table}
    \centering
    \begin{tabular}{|c|c|c||c|c|c|}
    \hline
    \multicolumn{3}{|c||}{Spatial convergence}  & \multicolumn{3}{|c|}{Temporal convergence} \\
    \hline
    $h$&Error&Rate&$\tau$&Error&Rate\\
    \hline
    2$\pi$/16&3.026E-2&-&0.02&4.336E-2&-\\
    \hline
    2$\pi$/32&7.635E-3&1.99&0.02/2&2.246E-2&0.95\\
    \hline
    2$\pi$/64&1.918E-3&1.99&0.02/4&1.086E-2&1.05\\  
    \hline
    2$\pi$/128&4.853E-4&1.98&0.02/8&4.741E-3&1.20\\
    \hline
    2$\pi$/256&1.282E-4&1.92&0.02/16&1.595E-3&1.57\\
    \hline

    \end{tabular}
    \caption{Errors and convergence rates.}
    \label{ta:rate}
\end{table}

\begin{figure}[htbp]
    \centering
    \begin{subfigure}[b]{0.49\textwidth}
        \centering
        \includegraphics[width=\textwidth]{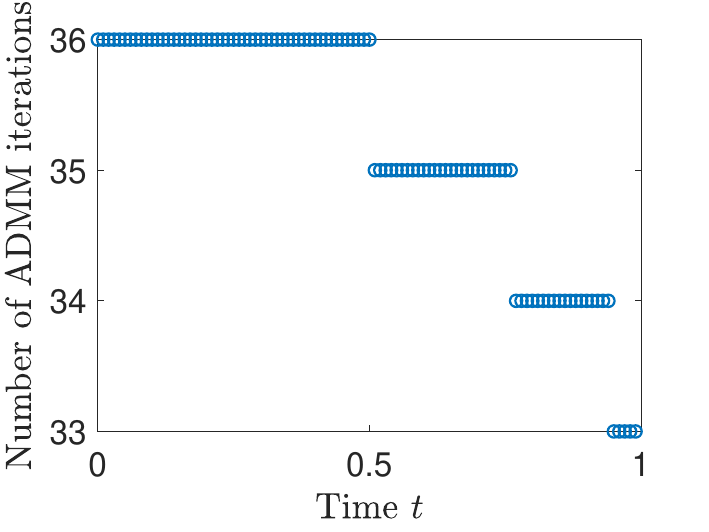}
        \caption{The number of iterations over time $t$.}
        \label{fig:11}
    \end{subfigure}
    \hfill
    \begin{subfigure}[b]{0.49\textwidth}
        \centering
        \includegraphics[width=\textwidth]{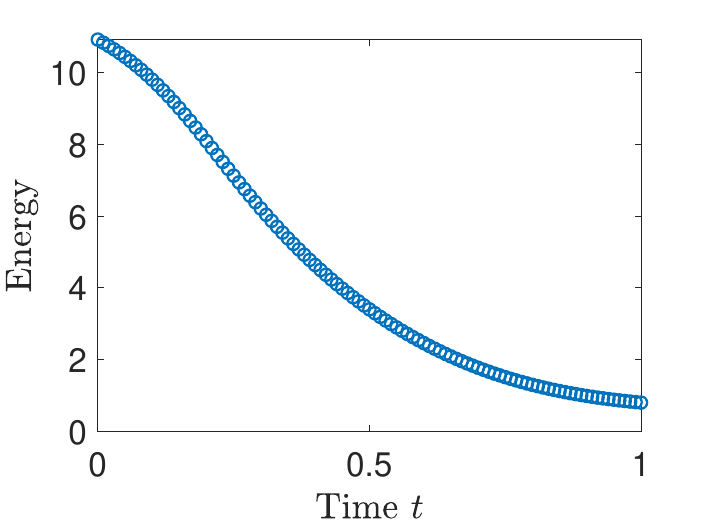}
        \caption{The energy over time $t$.}
        \label{fig:12}
    \end{subfigure}
    
    \vspace{0.5cm}
    
    \begin{subfigure}[b]{0.49\textwidth}
        \centering
        \includegraphics[width=\textwidth]{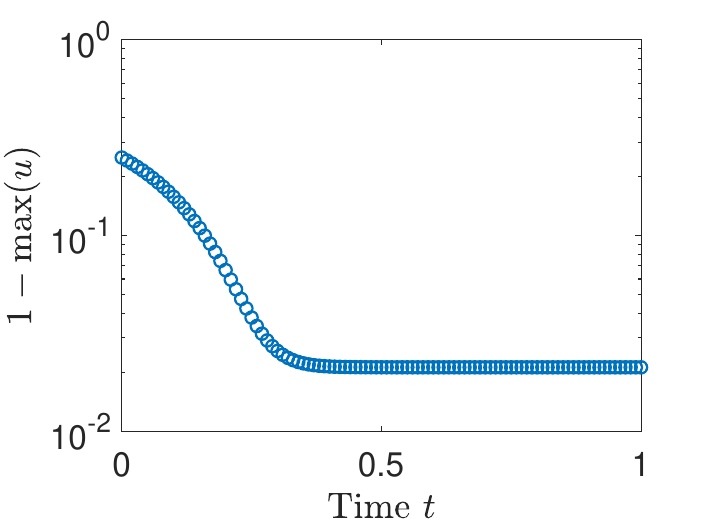}
        \caption{$1-\max (u)$ over time $t$.}
        \label{fig:13}
    \end{subfigure}
    \hfill
    \begin{subfigure}[b]{0.49\textwidth}
        \centering
        \includegraphics[width=\textwidth]{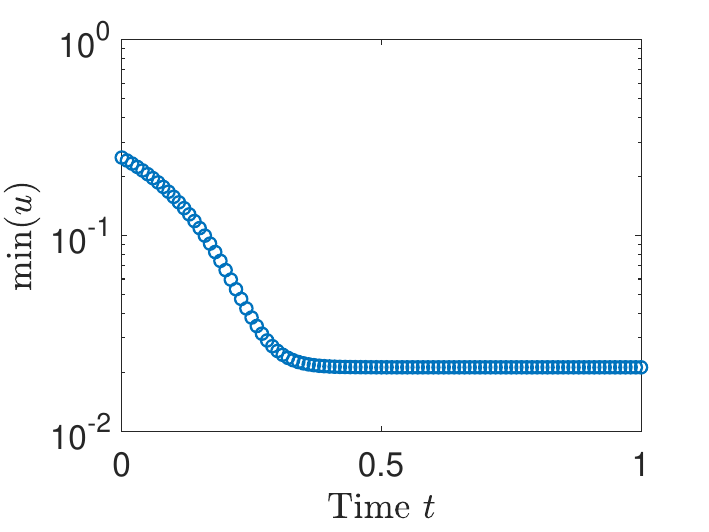}
        \caption{$\min (u)$ over time $t$}
        \label{fig:14}
    \end{subfigure}
    \vspace{0.5cm}

    \begin{subfigure}[b]{0.49\textwidth}
        \centering
        \includegraphics[width=\textwidth]{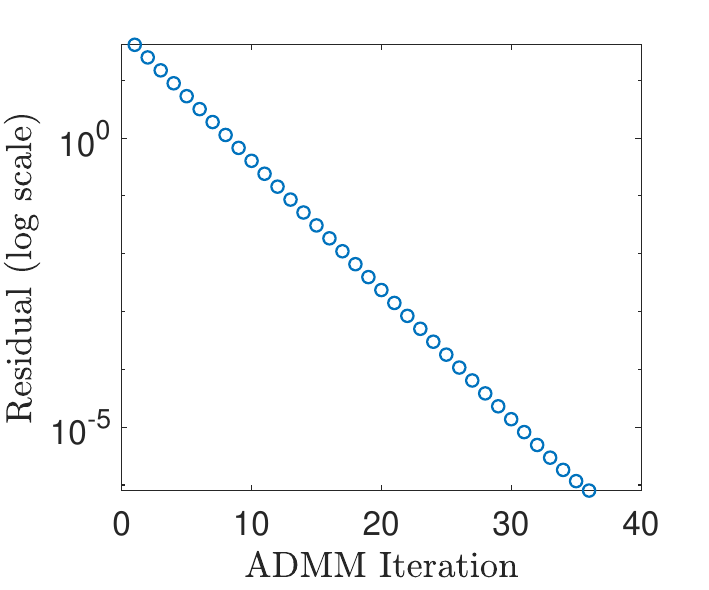}
        \caption{The primal residual in the second time step.}
        \label{fig:15}
    \end{subfigure}
    
    \caption{Convergence test}
    \label{fig:1}
\end{figure}

\begin{figure}[htbp]
    \centering
    \includegraphics[width=\textwidth]{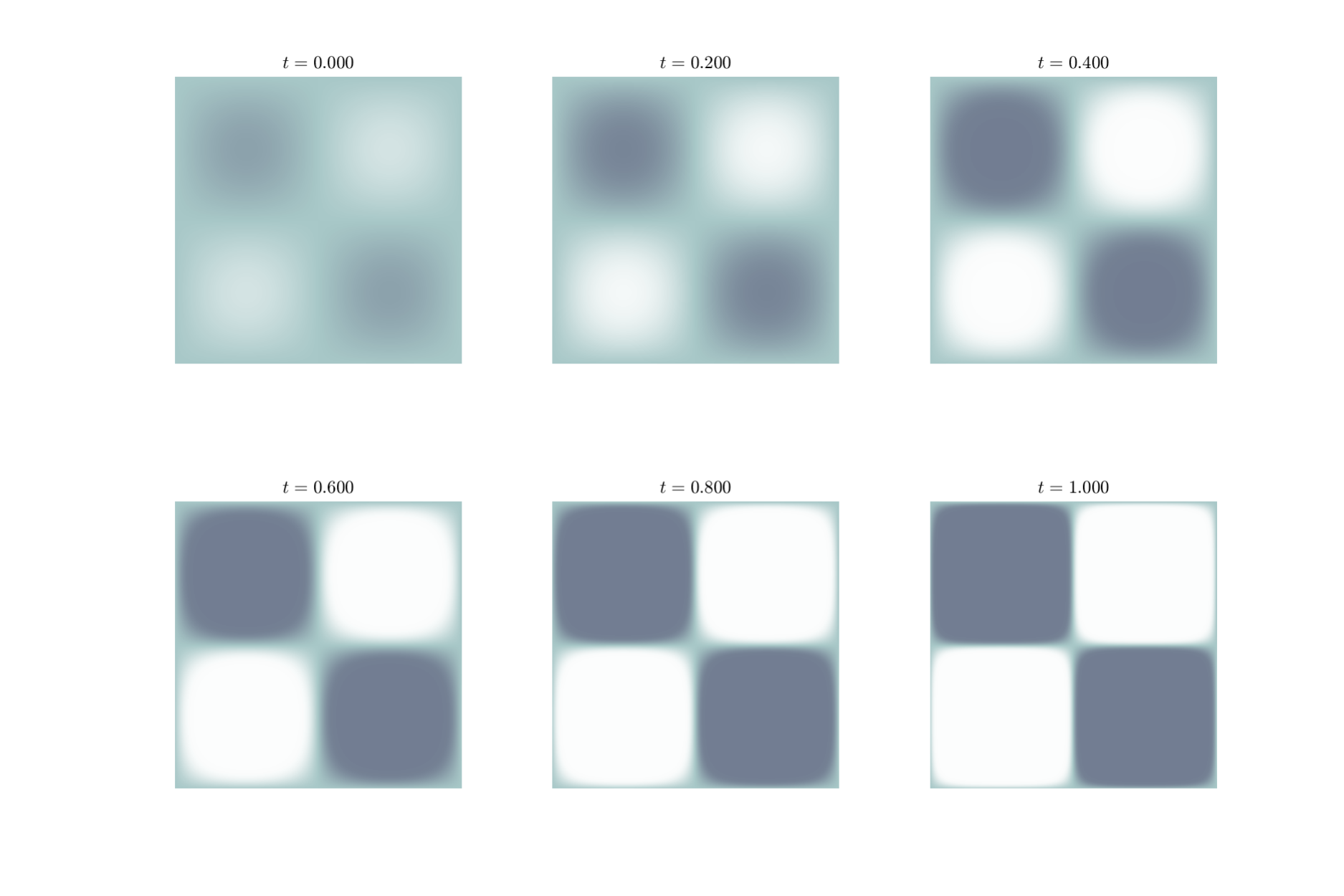}
    \caption{The dynamical evolution over time $t$.}
    \label{fig:16}
    \hfill
\end{figure}

\subsection{Two-dimensional simulation}
\label{subsec:2d}
In this numerical experiment, we show a two-dimensional simulation process here.

The following parameters are adopted in this example:
\begin{equation}
L=2.0,\ N=512,\ \epsilon=0.05,\ \tau=0.1,\ T=10.
\end{equation}
The parameter $\theta$ is chosen as the variable parameter in this experiment, which is taken as $\theta=3.0,4.0,5.0$.
The ability to utilize an exceptionally large time step of $\tau=0.1$ explicitly demonstrates the unconditional stability of our method, allowing for robust long-time simulations without the severe time-step constraints typical of explicit schemes.

The initial condition is
\begin{equation}
    u_0(x,y)=0.01+0.98\,\text{rand}(x,y).
\end{equation}
where $\text{rand}(x, y)$ is the function producing the random numbers in $(0,1)$. 

We set the convergence criterion for \Cref{alg:ADMM} as:
\begin{equation}
\max\{r,s\}\leq \gamma=10^{-8}.
\end{equation}

For $\theta=3.0$, the number of iterations, the energy, $1-\max (u)$ and $\min (u)$ over time $t$ are given in \Cref{fig:2_1}.
The dynamic evolution of the phase variable is depicted in \Cref{fig:2_2} for different choices of $\theta$.

\begin{figure}[htbp]
    \centering
    \begin{subfigure}[b]{0.49\textwidth}
        \centering
        \includegraphics[width=\textwidth]{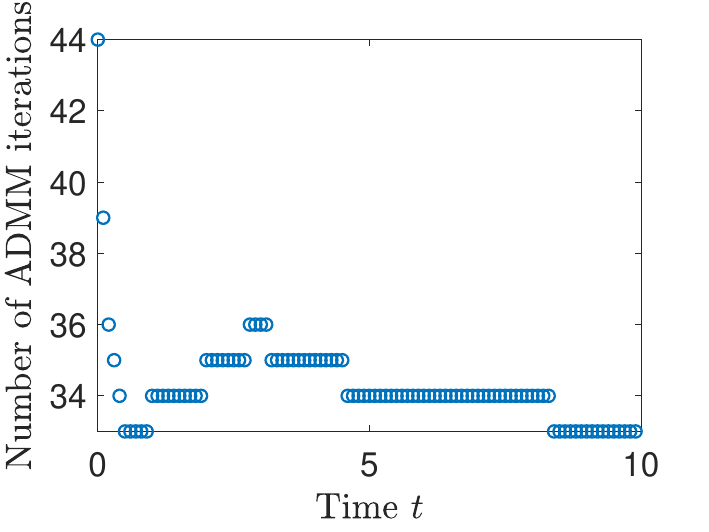}
        \caption{The number of iterations over time $t$.}
        \label{fig:21}
    \end{subfigure}
    \hfill
    \begin{subfigure}[b]{0.49\textwidth}
        \centering
        \includegraphics[width=\textwidth]{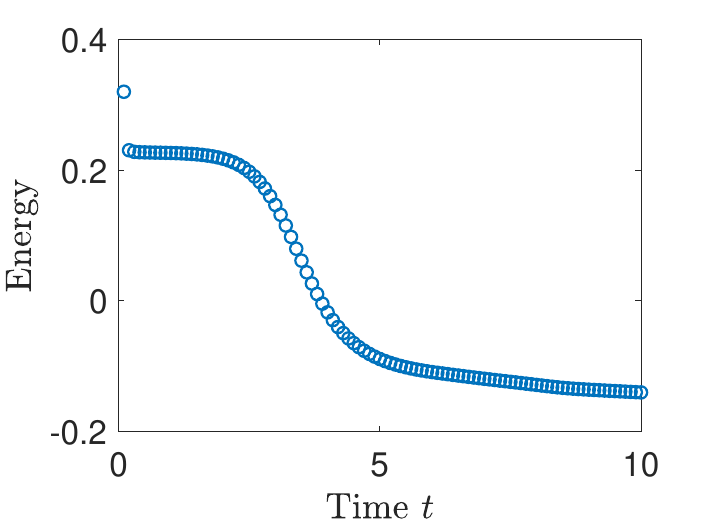}
        \caption{The energy over time $t$.}
        \label{fig:22}
    \end{subfigure}
    
    \vspace{0.5cm}
    
    \begin{subfigure}[b]{0.49\textwidth}
        \centering
        \includegraphics[width=\textwidth]{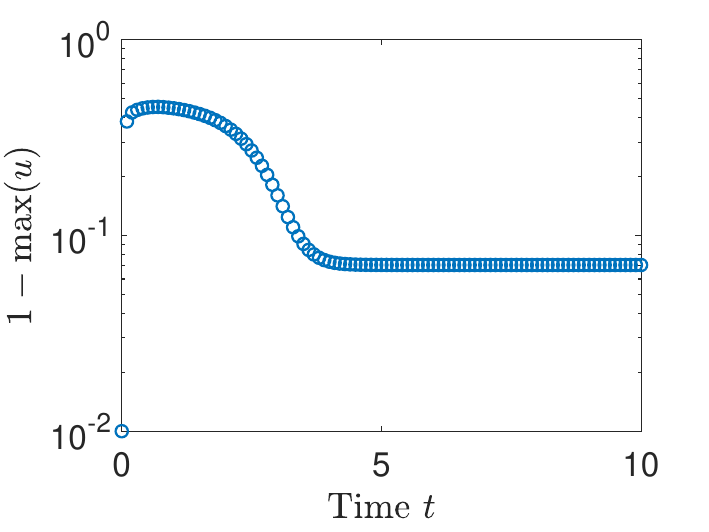}
        \caption{$1-\max (u)$ over time $t$.}
        \label{fig:23}
    \end{subfigure}
    \hfill
    \begin{subfigure}[b]{0.49\textwidth}
        \centering
        \includegraphics[width=\textwidth]{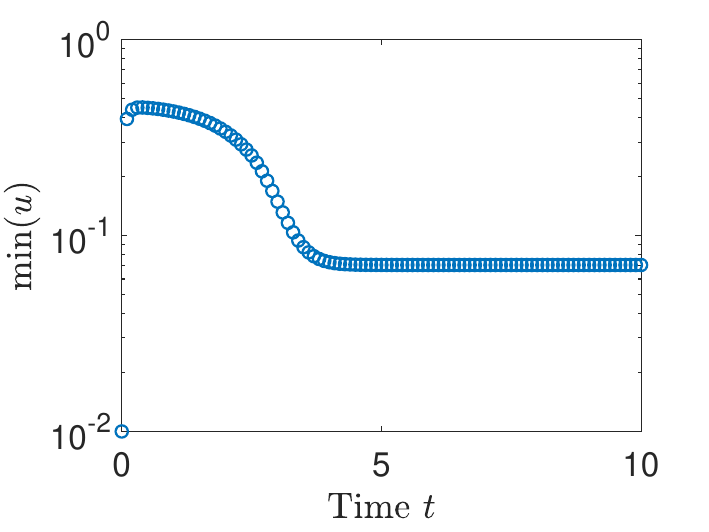}
        \caption{$\min (u)$ over time $t$}
        \label{fig:24}
    \end{subfigure}
    \vspace{0.5cm}
    
    \caption{2D simulation}
    \label{fig:2_1}
\end{figure}

\begin{figure}[htbp]
    \centering
    \begin{subfigure}[b]{1.00\textwidth}
        \centering
        \includegraphics[width=\textwidth]{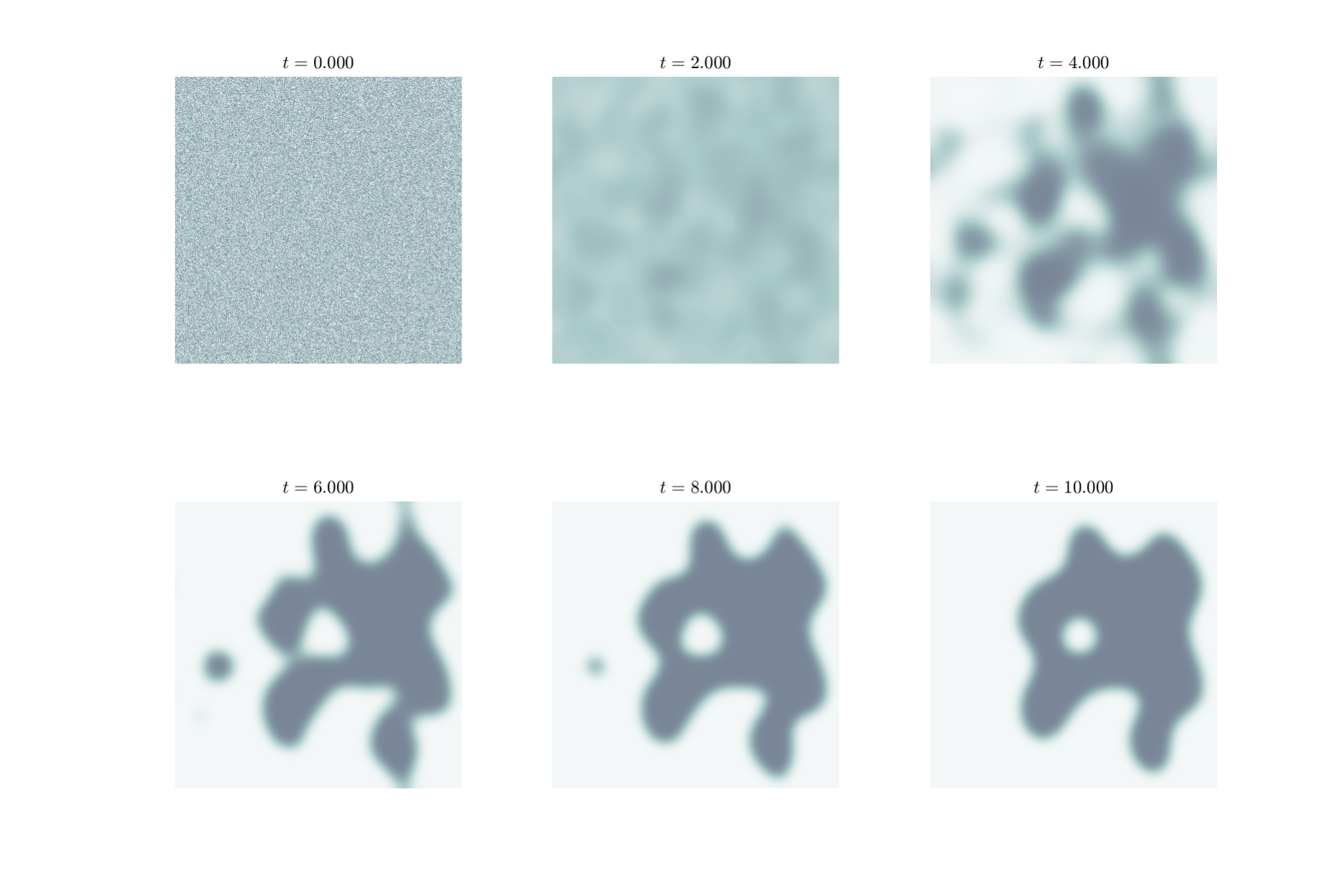}
        \caption{ $\theta=3.0$}
        \label{fig:25}
    \end{subfigure}
    \caption{2D simulation: the dynamical evolution over time $t$.}
    \label{fig:2_2}
\end{figure}

\begin{figure}[htbp]
    \ContinuedFloat
    \centering
    \begin{subfigure}[b]{1.00\textwidth}
        \centering
        \includegraphics[width=\textwidth]{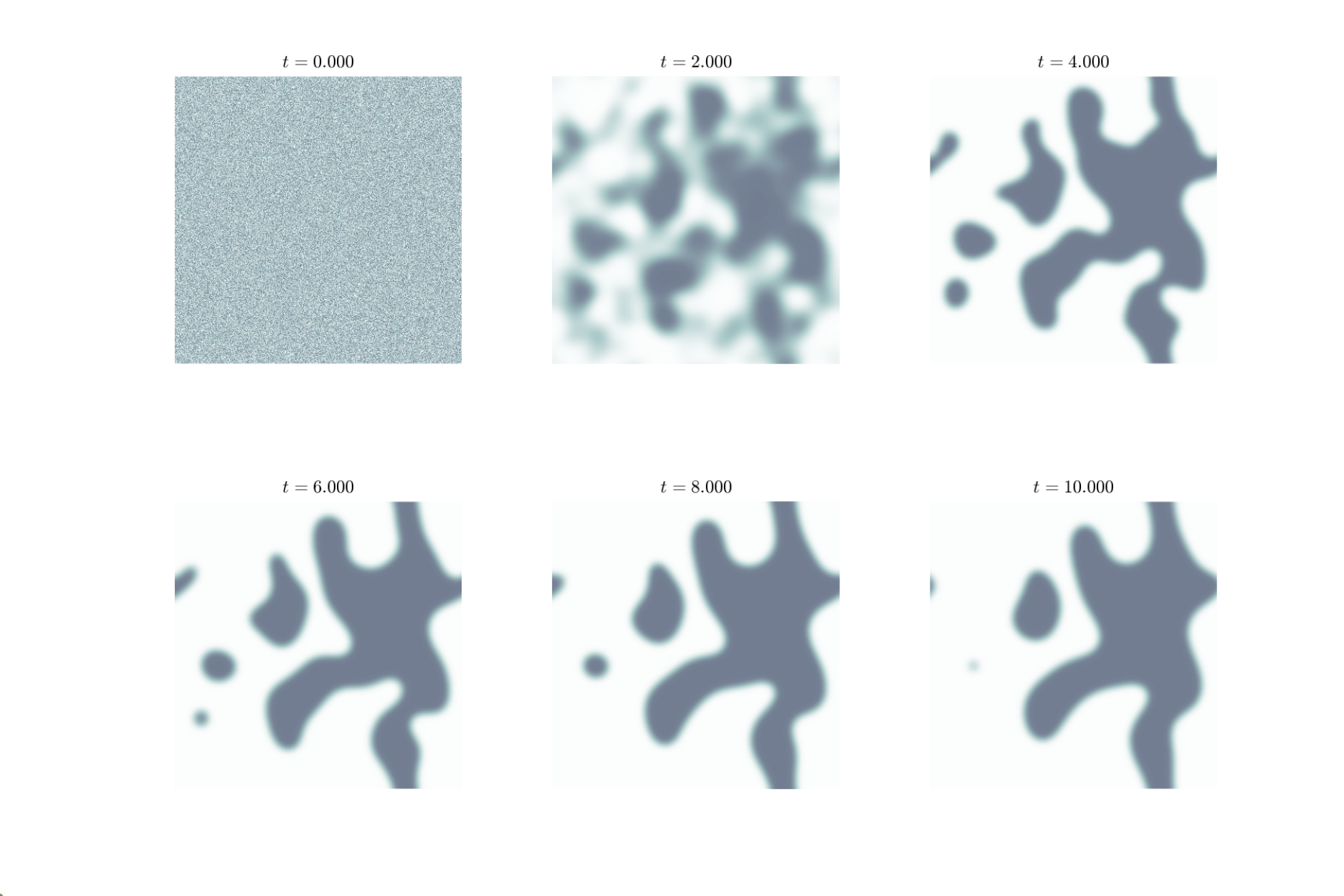}
        \caption{$\theta=4.0$}
        \label{fig:26}
    \end{subfigure}
    \caption{2D simulation: the dynamical evolution over time $t$.}
\end{figure}

\begin{figure}[htbp]
    \ContinuedFloat
    \centering
    \begin{subfigure}[b]{1.00\textwidth}
        \centering
        \includegraphics[width=\textwidth]{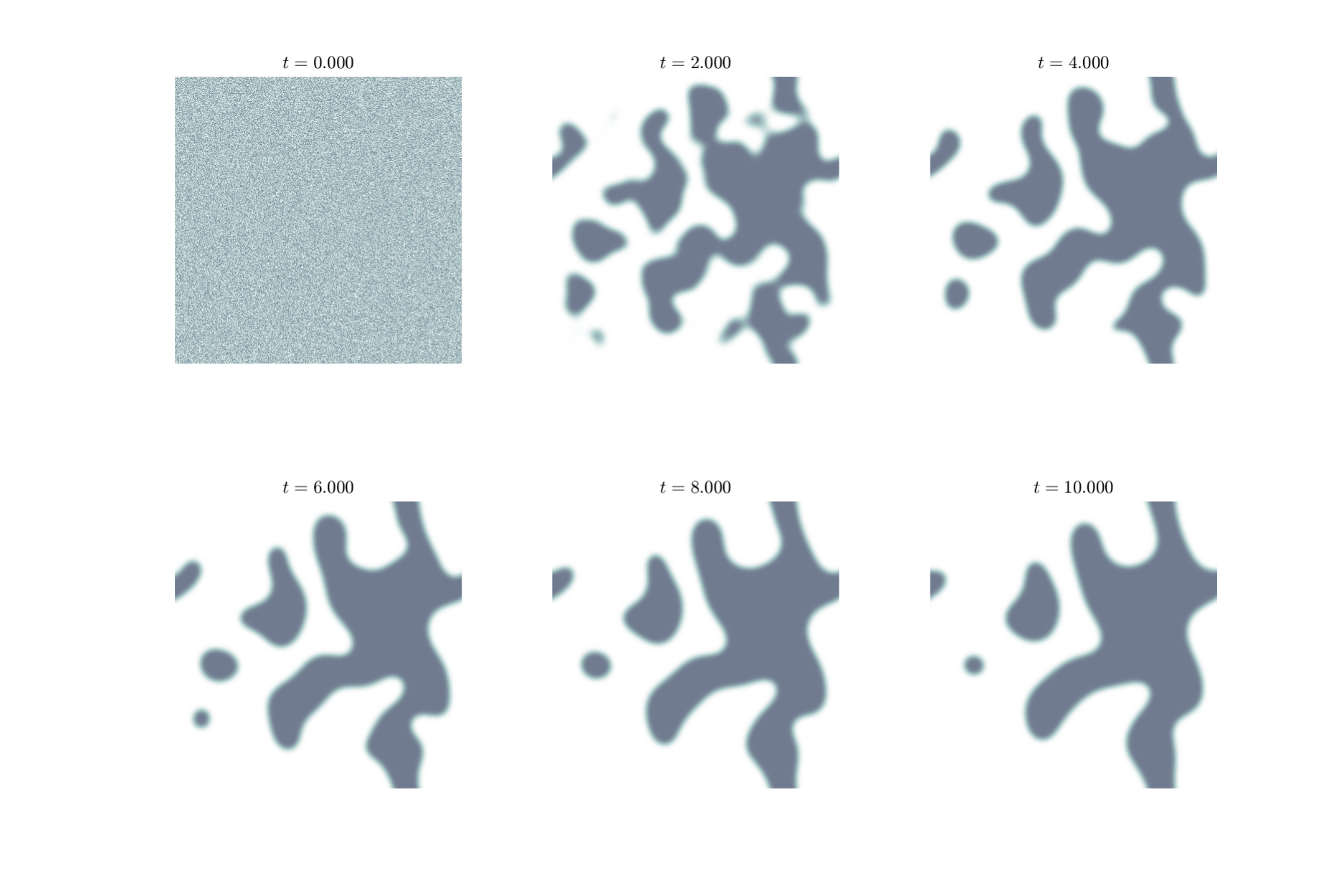}
        \caption{$\theta=5.0$}
        \label{fig:27}
    \end{subfigure}
    \caption{2D simulation: the dynamical evolution over time $t$.}
\end{figure}

\subsection{Three-dimensional simulation}
\label{subsec:3d}
We present a numerical simulation experiment in a three-dimensional setting here. The parameters we adopted are 
\begin{equation}
L=1.0,\ N=64,\ \theta=4.0,\ \tau=0.01,\ T=0.5.
\end{equation}
The parameter $\epsilon$ is chosen as the variable parameter in this experiment, which is taken as $\epsilon=0.05,0.10,0.15$.

The initial condition is
\begin{equation}
    u_0(x,y,z)=0.45+0.10\,\text{rand}(x,y,z).
\end{equation}
where $\text{rand}(x,y,z)$ is the function producing the random numbers in $(0,1)$. 

The convergence criterion is:
\begin{equation}
\max\{r,s\}\leq \gamma=10^{-8}.
\end{equation}

\begin{figure}[htbp]
    \centering
    \begin{subfigure}[b]{0.49\textwidth}
        \centering
        \includegraphics[width=\textwidth]{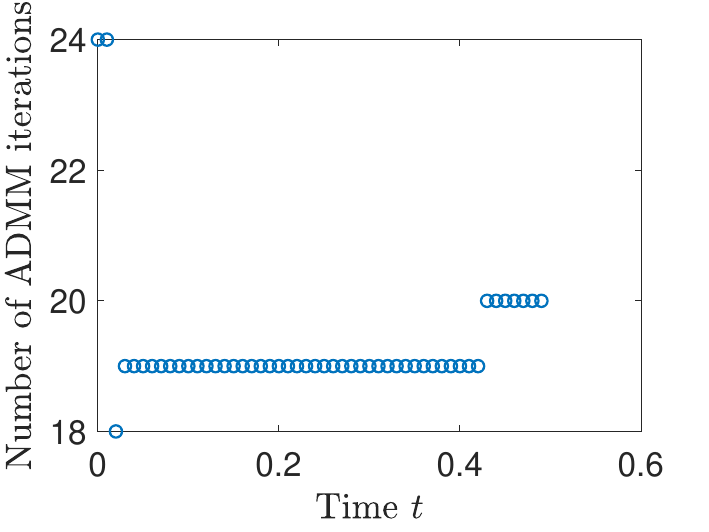}
        \caption{The number of iterations over time $t$.}
        \label{fig:31}
    \end{subfigure}
    \hfill
    \begin{subfigure}[b]{0.49\textwidth}
        \centering
        \includegraphics[width=\textwidth]{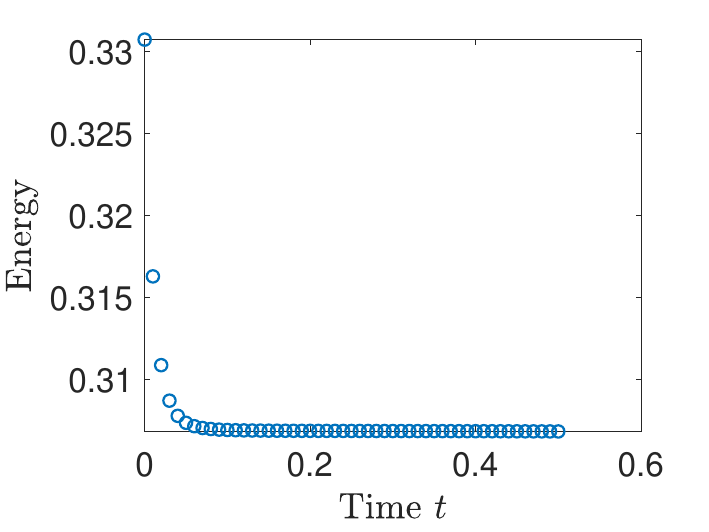}
        \caption{The energy over time $t$.}
        \label{fig:32}
    \end{subfigure}
    
    \vspace{0.5cm}
    
    \begin{subfigure}[b]{0.49\textwidth}
        \centering
        \includegraphics[width=\textwidth]{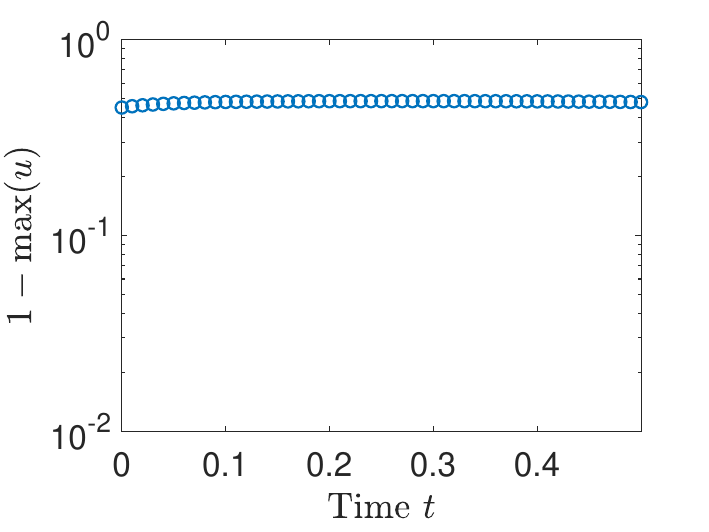}
        \caption{$1-\max (u)$ over time $t$.}
        \label{fig:33}
    \end{subfigure}
    \hfill
    \begin{subfigure}[b]{0.49\textwidth}
        \centering
        \includegraphics[width=\textwidth]{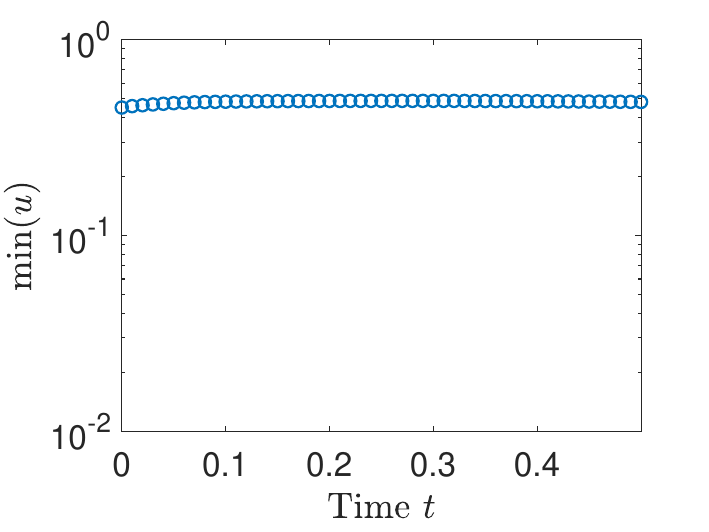}
        \caption{$\min (u)$ over time $t$}
        \label{fig:34}
    \end{subfigure}
    \vspace{0.5cm}
    
    \caption{3D simulation}
    \label{fig:3_1}
\end{figure}

\begin{figure}[htbp]
    \centering
    \begin{subfigure}[b]{1.00\textwidth}
        \centering
        \includegraphics[width=\textwidth]{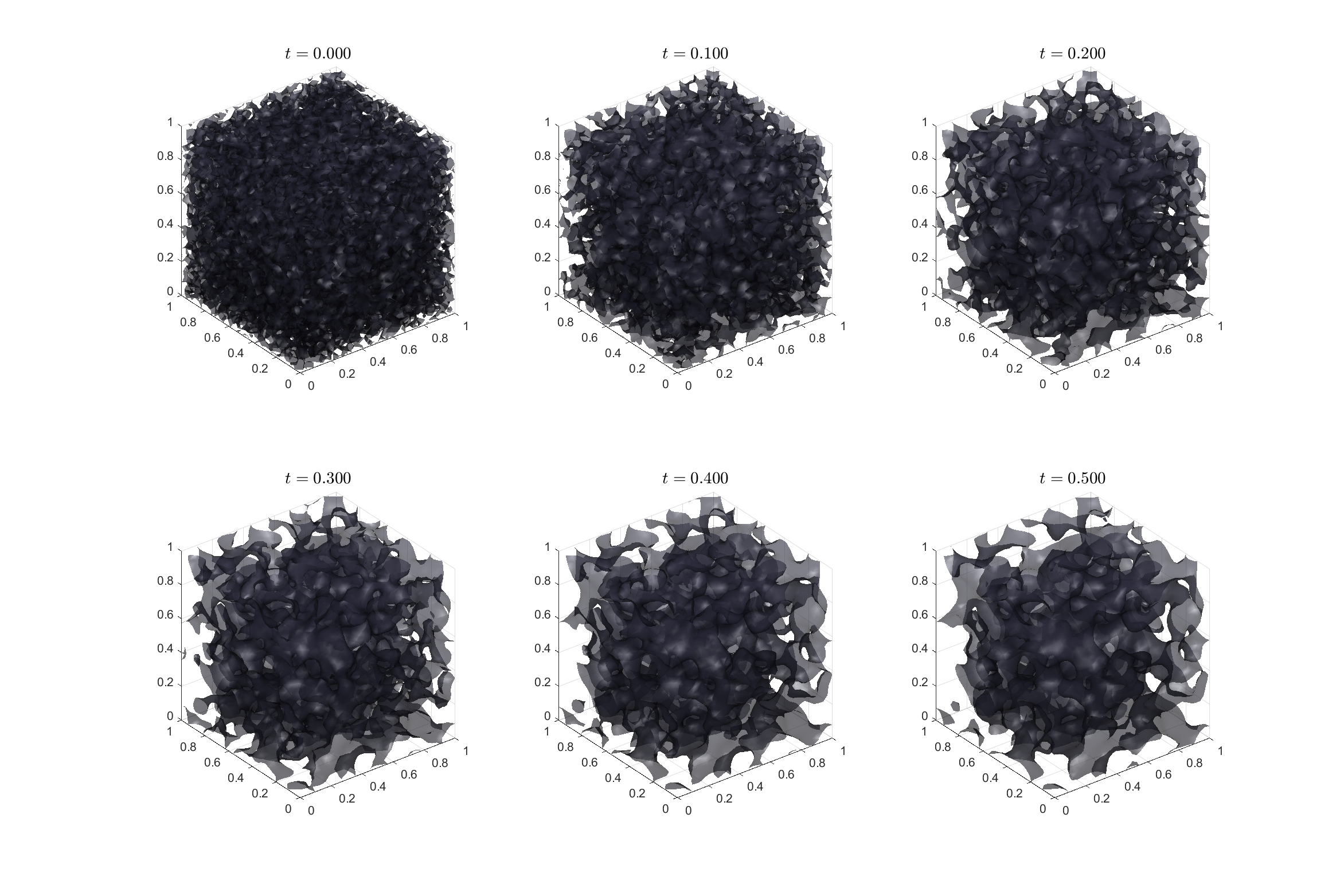}
        \caption{$\epsilon=0.05$}
        \label{fig:35}
    \end{subfigure}
    \caption{3D simulation: the dynamical evolution over time $t$.}
    \label{fig:3_2}
\end{figure}

\begin{figure}[htbp]
    \ContinuedFloat
    \centering
    \begin{subfigure}[b]{1.00\textwidth}
        \centering
        \includegraphics[width=\textwidth]{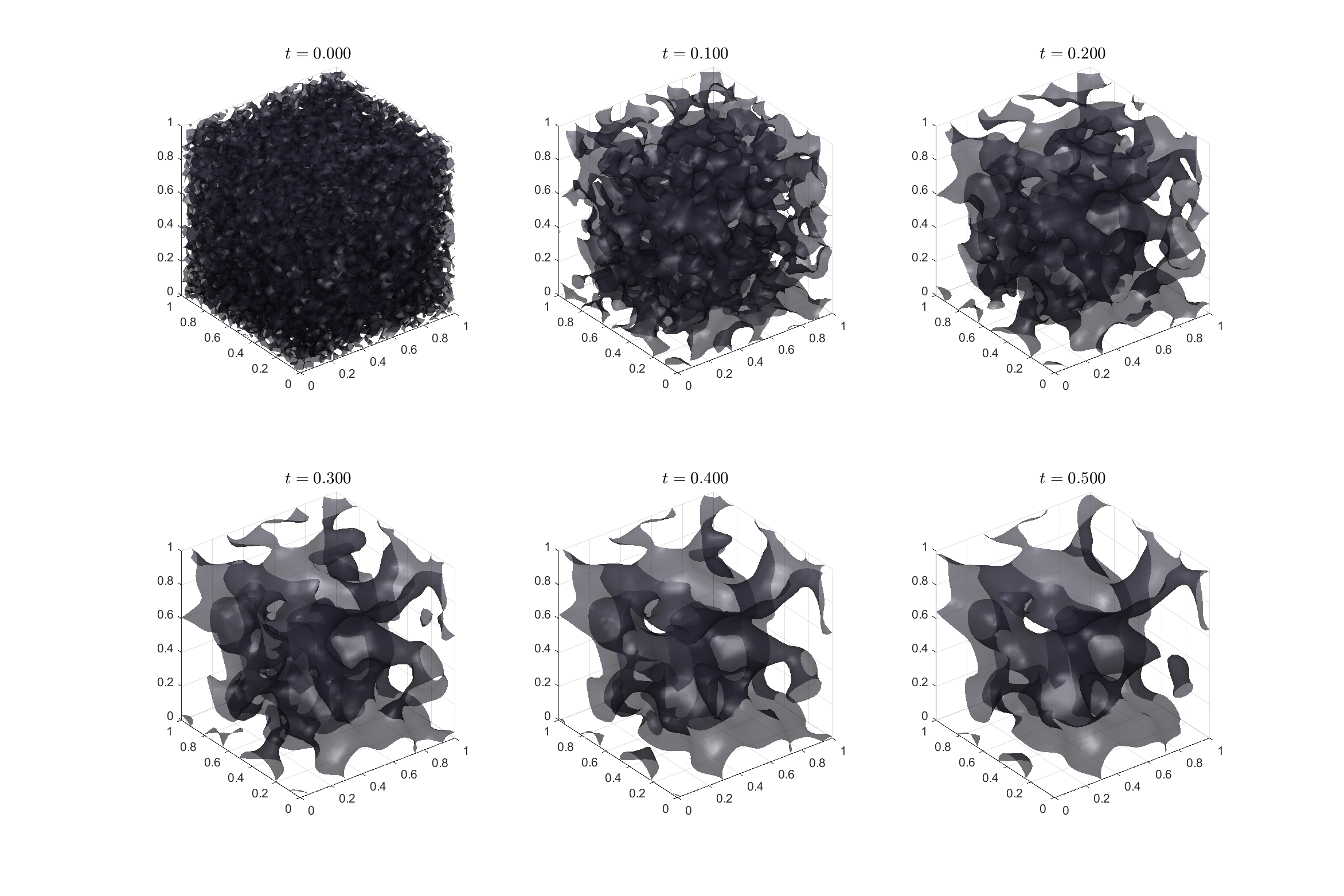}
        \caption{$\epsilon=0.10$}
        \label{fig:36}
    \end{subfigure}
    \caption{3D simulation: the dynamical evolution over time $t$.}
\end{figure}

\begin{figure}[htbp]
    \ContinuedFloat
    \centering
    \begin{subfigure}[b]{1.00\textwidth}
        \centering
        \includegraphics[width=\textwidth]{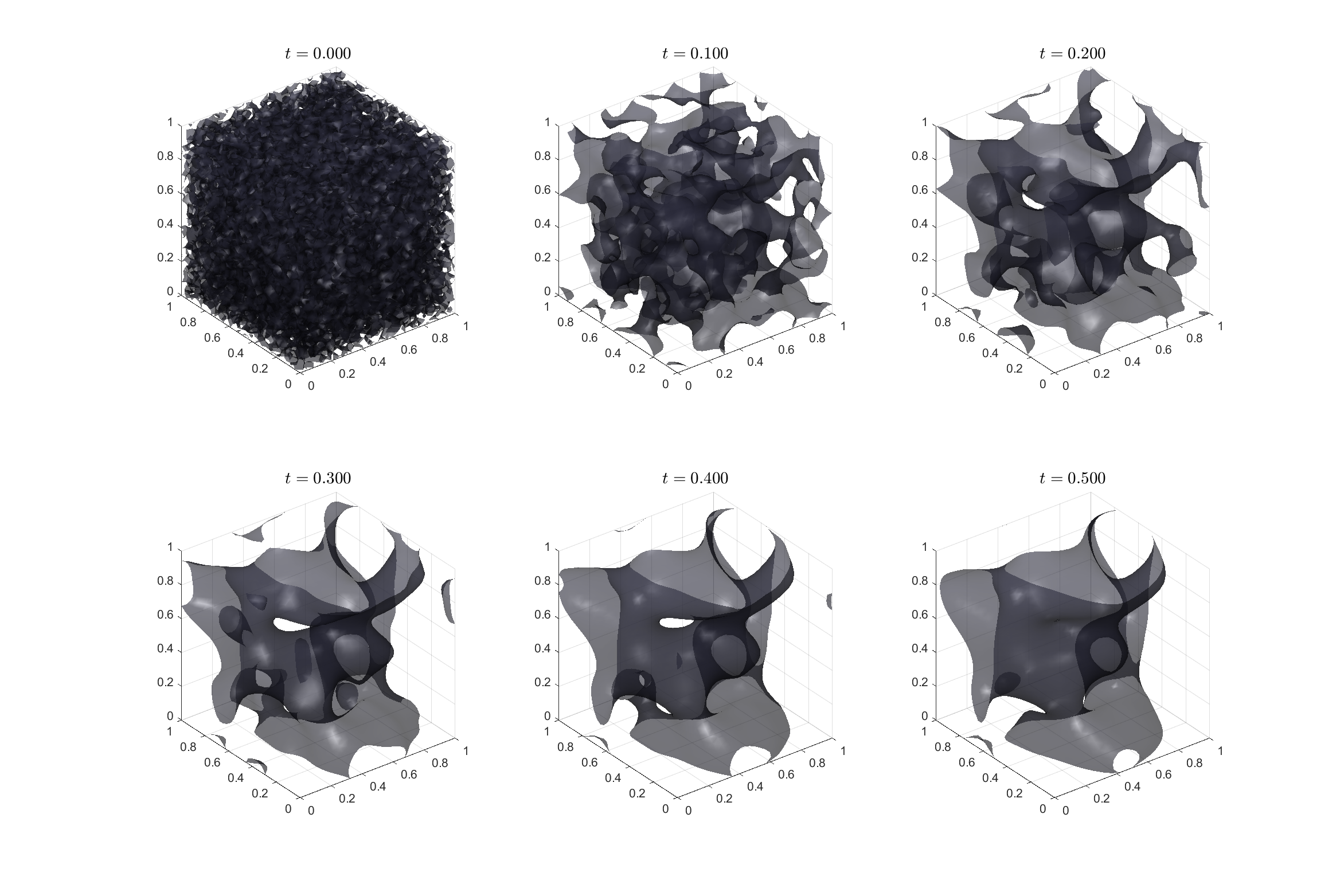}
        \caption{$\epsilon=0.15$}
        \label{fig:37}
    \end{subfigure}
    \caption{3D simulation: the dynamical evolution over time $t$.}
\end{figure}

For $\epsilon=0.05$, the number of iterations, the energy, $1-\max (u)$ and $\min (u)$ over time $t$ are given in \Cref{fig:3_1}. 
\Cref{fig:3_2} shows the interface of $u=0.50$ by interpolating the numerical solution at different time steps from the three-dimensional simulation. 
The results indicate that the algorithm accurately captures the dynamical process of the two-phase interface with satisfactory precision,
which is consistent with the physical phenomenon of phase separation. In summary, the numerical result demonstrates the robustness of our solver in numerical simulation.

\section{Conclusion}
\label{section:conclusion}
In this paper, we present a novel iterative solver for the nonlinear system at each time step, which naturally arises from applying the convex splitting scheme to the Allen-Cahn equation with the Flory-Huggins potential.
We employ an ADMM framework to solve the resulting problem by reformulating the complex nonlinear system equivalently as a convex optimization problem.
Theoretically, we rigorously establish the well-posedness and discrete energy stability of the convex splitting scheme.
More importantly, we provide a mathematical proof of the unconditional linear convergence for the proposed ADMM solver.
This theoretical breakthrough liberates the solver from strict separation assumptions or time-step constraints, effectively overcoming the severe numerical challenges posed by the logarithmic singularities.
The effectiveness and robustness of the algorithm are thoroughly demonstrated through numerical examples, which fully substantiate our theoretical predictions.
Furthermore, while the current work focuses on the Allen-Cahn equation with periodic boundary conditions, the proposed solver framework inherently possesses broad applicability. It can be readily extended to other boundary conditions, such as Neumann and Dirichlet types, and adapted to more general spatial discretizations, including the finite element method \cite{yuan2022second}.


Despite these promising results, several directions merit further investigation:
(a) Constructing effective initial guesses (e.g., via extrapolation) to further reduce iteration counts in long-time evolution simulations.
(b) Extending the framework to second-order convex splitting schemes. While our solver is readily applicable, rigorously preserving the original continuous energy stability for second-order schemes remains an open theoretical challenge in the field \cites{gomez2011provably,yang2016linear} that warrants future study.
(c) Conducting comprehensive benchmark comparisons. Although our theoretical convergence is rigorously validated, evaluating the empirical computational efficiency against existing solvers remains an important ongoing task.

\bibliographystyle{amsplain}
\bibliography{article.bib}

\end{document}